\documentclass[11pt, reqno]{amsart}
\usepackage{amssymb}
\usepackage{eucal}
\usepackage{amsmath}
\usepackage{amscd}
\usepackage[dvips]{color}
\usepackage{multicol}
\usepackage[all]{xy}           
\usepackage{graphicx}
\usepackage{color}
\usepackage{colordvi}
\usepackage{xspace}
\usepackage{ulem}
\usepackage{cancel}
\usepackage{tikz}
\usepackage{longtable}
\usepackage{multirow}

\usepackage{ifpdf}
\ifpdf
    \usepackage[colorlinks,final,hyperindex]{hyperref} \else
    \usepackage[colorlinks,final,hyperindex,hypertex]{hyperref}
\fi

\begin{document}
\newcommand {\emptycomment}[1]{} 

\newcommand{\tabincell}[2]{\begin{tabular}{@{}#1@{}}#2\end{tabular}}

\newcommand{\nc}{\newcommand}
\newcommand{\delete}[1]{}
\nc{\mfootnote}[1]{\footnote{#1}} 
\nc{\todo}[1]{\tred{To do:} #1}

\delete{
\nc{\mlabel}[1]{\label{#1}}  
\nc{\mcite}[1]{\cite{#1}}  
\nc{\mref}[1]{\ref{#1}}  
\nc{\meqref}[1]{\eqref{#1}} 
\nc{\bibitem}[1]{\bibitem{#1}} 
}

\nc{\mlabel}[1]{\label{#1}  
{\hfill \hspace{1cm}{\bf{{\ }\hfill(#1)}}}}
\nc{\mcite}[1]{\cite{#1}{{\bf{{\ }(#1)}}}}  
\nc{\mref}[1]{\ref{#1}{{\bf{{\ }(#1)}}}}  
\nc{\meqref}[1]{\eqref{#1}{{\bf{{\ }(#1)}}}} 
\nc{\mbibitem}[1]{\bibitem[\bf #1]{#1}} 

\newtheorem{thm}{Theorem}[section]
\newtheorem{lem}[thm]{Lemma}
\newtheorem{cor}[thm]{Corollary}
\newtheorem{pro}[thm]{Proposition}
\newtheorem{conj}[thm]{Conjecture}
\theoremstyle{definition}
\newtheorem{defi}[thm]{Definition}
\newtheorem{ex}[thm]{Example}
\newtheorem{rmk}[thm]{Remark}
\newtheorem{pdef}[thm]{Proposition-Definition}
\newtheorem{condition}[thm]{Condition}



\title[Anti-pre-Lie algebras]
{Anti-pre-Lie algebras, Novikov algebras and commutative
2-cocycles on Lie algebras}

\author{Guilai Liu}
\address{Chern Institute of Mathematics \& LPMC, Nankai University, Tianjin 300071, China}
\email{1120190007@mail.nankai.edu.cn}

\author{Chengming Bai}
\address{Chern Institute of Mathematics \& LPMC, Nankai University, Tianjin 300071, China }
\email{baicm@nankai.edu.cn}


\begin{abstract}

 We introduce the
notion of anti-pre-Lie algebras as the underlying algebraic
structures of nondegenerate commutative 2-cocycles which are the
``symmetric" version of symplectic forms on Lie algebras. They can
be characterized as a class of Lie-admissible algebras whose
negative left multiplication operators make  representations of
the commutator Lie algebras. We observe that there is  a clear
analogy between
 anti-pre-Lie algebras and pre-Lie algebras by comparing them in terms of
several aspects. Furthermore, it is unexpected that a subclass of
anti-pre-Lie algebras, namely admissible Novikov algebras,
correspond to Novikov algebras in terms of $q$-algebras.
Consequently, there is a construction of admissible Novikov
algebras from commutative associative algebras with derivations or
more generally, admissible pairs. The correspondence extends to
the level of Poisson type structures, leading to the introduction
of the notions of anti-pre-Lie Poisson algebras and admissible
Novikov-Poisson algebras, whereas the latter correspond to
Novikov-Poisson algebras.

\end{abstract}


\subjclass[2010]{
17A36,  
17A40,  
17B10, 
17B40, 
17B60, 
17B63,  
17D25.  
}

\keywords{anti-pre-Lie algebra; commutative 2-cocycle; pre-Lie
algebra; Novikov algebra}

\maketitle


\tableofcontents

\allowdisplaybreaks

\section{Introduction}

This paper aims to introduce the notion of anti-pre-Lie algebras
and then study the relationships between them and the related
structures such as anti-$\mathcal O$-operators, commutative
2-cocycles on Lie algebras and Novikov algebras.

At first we recall the notion of commutative 2-cocycles on Lie algebras.

\begin{defi} (\cite{Dzh})
A \textbf{commutative 2-cocycle} $\mathcal B$ on a Lie algebra
$(\frak g,[-,-])$ is a symmetric bilinear form such that
\begin{equation}\label{eq:2-cocycle}
\mathcal{B}([x,y],z)+\mathcal{B}([y,z],x)+\mathcal{B}([z,x],y)=0,\forall
x,y,z\in \frak g.
\end{equation}
\end{defi}

Commutative 2-cocycles appear in the study of non-associative
algebras satisfying certain skew-symmetric identities
(\cite{Dzh09}), and also in the description of the second
cohomology of current Lie algebras (\cite{Zu}). In the
nondegenerate cases, they are the ``symmetric" version of
symplectic forms on Lie algebras, whereas the latter are the
skew-symmetric bilinear forms satisfying Eq.~(\ref{eq:2-cocycle}).

On the other hand, in the symmetric case, there is the {\bf
invariant} bilinear form $\mathcal B$ on a Lie algebra $(\frak g,[-,-])$
in the sense that
\begin{equation}
\mathcal B([x,y],z)=\mathcal B(x,[y,z]),\;\;\forall x,y,z\in \frak g.
\end{equation}
The Killing forms on Lie algebras are the examples of invariant
bilinear forms. Note the invariance and a commutative 2-cocycle
are ``inconsistent" in the sense that $\mathcal B$ is both
invariant and a commutative 2-cocycle if and only if $3\mathcal
B(x,[y,z])=0$ for all $x,y,z\in \frak g$, and hence in particular,
in the nondegenerate case and if the characteristic of the base
field is not 3, it holds if and only if the Lie algebra $(\frak
g,[-,-])$ is abelian.

Furthermore, a class of Lie algebras are obtained from commutative
associative algebras with derivations as follows. Let $P$ be a
derivation on a commutative associative algebra $(A,\cdot)$. Then
$(A,[-,-])$ is a Lie algebra with the bilinear operation
$[-,-]:A\otimes A\rightarrow A$  given by
\begin{equation}\label{eq:Lie algebras form differential commutative associative algebras}
[x,y]=P(x)\cdot y-x\cdot P(y),\forall x,y\in A.
\end{equation}
Such Lie algebras are called Witt type Lie algebras (\cite{SXZ,Xu}) and
provide ``typical examples" in
several topics in mathematics, such as Jacobi algebras
(\cite{AM}), generalized Poisson algebras (\cite{CK}),
$F$-manifold algebras (\cite{Dot}), contact brackets (\cite{MZ})
and transposed Poisson algebras (\cite{Bai2020}). On the
commutative associative algebra $(A,\cdot)$, there is a natural
``invariance" for symmetric bilinear forms in the sense that
\begin{equation}\label{eq:inv-asso}
\mathcal B(x\cdot y,z)=\mathcal B(x,y\cdot z),\;\;\forall x,y,z\in A.
\end{equation}
In particular, if $(A,\cdot)$ has a nondegenerate symmetric
invariant bilinear form, then it is exactly a symmetric
(commutative) Frobenius algebra (\cite{Fro}). It is natural to
consider the bilinear form on the Lie algebra $(A,[-,-])$ defined
by Eq.~(\ref{eq:Lie algebras form differential commutative
associative algebras}) induced from an invariant bilinear form
$\mathcal B$ on $(A,\cdot)$: in general $\mathcal B$ is not
invariant on $(A,[-,-])$, but a commutative 2-cocycle on
$(A,[-,-])$ (see Corollary~\ref{cor:construction}).
Hence the study of commutative 2-cocycles deserves
attracting more attention.
It is known that there is an underlying algebraic structure on a
symplectic form on a Lie algebra, namely, a pre-Lie algebra, that
is, a pre-Lie algebra can be induced from a symplectic form on a
Lie algebra (\cite{Chu}). Pre-Lie algebras arose from the study of
affine manifolds and affine structures on Lie groups (\cite{Ko}),
convex homogeneous cones (\cite{V}) and deformations and
cohomologies of associative algebras (\cite{Ger}) and then appear
in many fields in mathematics and physics (\cite{Bur}). Similarly,
we introduce a new kind of algebraic structures, called
anti-pre-Lie algebras, as the underlying algebraic structures of
nondegenerate commutative 2-cocycles on Lie algebras, that is,
there are anti-pre-Lie algebras induced from the latter.

Besides this ``similarity", anti-pre-Lie algebras have some other
properties which are analogue to the ones of pre-Lie algebras. In
fact, anti-pre-Lie algebras are characterized as a class of
Lie-admissible algebras whose negative left multiplication
operators make representations of the commutator Lie algebras,
whereas pre-Lie algebras are a class of Lie-admissible algebras
whose left multiplication operators make representations of the
commutator Lie algebras, justifying the notion of anti-pre-Lie
algebras. Consequently there are the constructions of
nondegenerate commutative 2-cocycles and symplectic forms on
semi-direct product Lie algebras from anti-pre-Lie algebras and
pre-Lie algebras (\cite{Ku1}) respectively. We introduce the
notion of an anti-$\mathcal O$-operator on a Lie algebra to
interpret anti-pre-Lie algebras, motivating from the notion of an
$\mathcal O$-operator introduced in \cite{Kup} as a natural
generalization of the classical Yang-Baxter equation in a Lie
algebra which is used to interpret pre-Lie algebras
(\cite{Bai2007}). The analogues also appear in the constructions
of examples from linear functions and symmetric bilinear forms
(\cite{Bai2004,SS}). On the other hand, an obvious difference
between the anti-pre-Lie algebras and the pre-Lie algebras is that
over the field of characteristic zero, the sub-adjacent Lie
algebras of the former can be simple (Example \ref{ex:anti-pre-Lie
algebra from sl(2)}), whereas there is not a compatible pre-Lie
algebra structure on a simple Lie algebra (\cite{Me}).

Furthermore, it is unexpected that a subclass of anti-pre-Lie
algebras, namely, admissible Novikov algebras, correspond to
Novikov algebras which are a subclass of pre-Lie algebras in terms
of $q$-algebras (\cite{Dzh09}). Novikov algebras were introduced
in connection with Hamiltonian operators in the formal variational
calculus (\cite{Gel}) and Poisson brackets of hydrodynamic type
(\cite{Bal}), independently from the study of pre-Lie algebras
themselves. They also correspond to certain vertex algebras and
Virasoro-like Lie algebras (\cite{BLP}). The relationship among
anti-pre-Lie algebras, admissible Novikov algebras, Novikov
algebras and pre-Lie algebras is summarized as follows.

$$\mbox{\{pre-Lie\} $\hookleftarrow$ \{Novikov\}
$\overset{2-algebra}{\underset{(-2)-algebra}{\rightleftharpoons}}$
\{admissible Novikov\} $\hookrightarrow$ \{anti-pre-Lie\}.}$$

The correspondence  between admissible Novikov algebras and
Novikov algebras provides a construction of anti-pre-Lie algebras
(admissible Novikov algebras) from examples of Novikov algebras,
in particular, from commutative associative algebras with
derivations (\cite{Bal,Fil1989,Xu1997}), or more generally,
admissible pairs. Such a construction especially coincides with
the anti-pre-Lie algebra structure induced from the aforementioned
nondegenerate commutative 2-cocycle on the Lie algebra $(A,[-,-])$
defined by Eq. (\ref{eq:Lie algebras form differential commutative
associative algebras}) under certain conditions.

Another consequence from the correspondence between admissible
Novikov algebras and Novikov algebras is that motivated from the
notion of Novikov-Poisson algebras (\cite{Xu1997}), the
correspondence is extended to the level of Poisson type structures
and hence we introduce the notions of admissible Novikov-Poisson
algebras and further anti-pre-Lie Poisson algebras. 
 Taking the
sub-adjacent Lie algebra of the anti-pre-Lie algebra in an
anti-pre-Lie Poisson algebra gives a transposed Poisson algebra
(\cite{Bai2020}) and conversely, a transposed Poisson algebra with
a nondegenerate symmetric bilinear form which is invariant on the
commutative associative algebra and a commutative 2-cocycle on the
Lie algebra induces an anti-pre-Lie Poisson algebra.
Moreover, there is also a tensor theory for
anti-pre-Lie Poisson algebras.

This paper is organized as follows.
In Section 2,  we first introduce the notion of anti-pre-Lie
algebras as a class of Lie-admissible algebras whose negative left
multiplication operators make  representations of the commutator
Lie algebras. We give some examples and the classification in
dimension 2. Then we introduce the notion of
anti-$\mathcal{O}$-operators to interpret anti-pre-Lie algebras.
There is a close relationship between anti-pre-Lie algebras and
commutative 2-cocycles on Lie algebras. That is, the former can be
induced from the latter in the nondegenerate case, whereas there
is a natural construction of the latter on the semi-direct product
Lie algebras induced from the former. Finally we summarize these
results to exhibit a clear analogy between anti-pre-Lie algebras
and pre-Lie algebras. In Section 3, we first introduce the notion
of admissible Novikov algebras as a subclass of anti-pre-Lie
algebras, corresponding to Novikov algebras in terms of
$q$-algebras. Then we give the constructions of Novikov algebras
and the corresponding admissible Novikov algebras from commutative
associative algebras with admissible pairs including derivations.
Finally we extend the correspondence between Novikov algebras and
admissible Novikov algebras to the level of Poisson type
structures, and hence introduce the notions of anti-pre-Lie
Poisson algebras and admissible Novikov-Poisson algebras. The
relationships with transposed Poisson algebras and Novikov-Poisson
algebras as well as a tensor theory are given.


Throughout this paper, unless otherwise specified, all vector spaces are assumed to be
finite-dimensional over a field $\mathbb{F}$ of characteristic 0.

\section{Anti-pre-Lie algebras, anti-$\mathcal{O}$-operators and commutative 2-cocycles on Lie algebras}

We introduce the notion of anti-pre-Lie algebras and exhibit their
properties in several aspects, giving a clear analogy between them
and pre-Lie algebras.



\subsection{Notions and examples of anti-pre-Lie algebras}

We introduce the notion of anti-pre-Lie algebras as a class of
Lie-admissible algebras whose negative left multiplication
operators make  representations of the commutator Lie algebras. We
give some basic examples as well as the construction of
anti-pre-Lie algebras from linear functions. The classification of
2-dimensional complex non-commutative anti-pre-Lie algebras is
also given.

\begin{defi}\label{defi:anti-pre-Lie algebras}
Let $A$ be a vector space with a bilinear operation $\circ: A\otimes A\rightarrow A$. $(A,\circ)$ is called an
\textbf{anti-pre-Lie algebra} if the following equations are
satisfied:
\begin{equation}\label{eq:defi:anti-pre-Lie algebras1}
x\circ(y\circ z)-y\circ(x\circ z)=[y,x]\circ z,
\end{equation}
\begin{equation}\label{eq:defi:anti-pre-Lie algebras2}
[x,y]\circ z+[y,z]\circ x+[z,x]\circ y=0,
\end{equation}
where
\begin{equation}\label{eq:defi:anti-pre-Lie algebras3}
[x,y]=x\circ y-y\circ x,
\end{equation}
for all $x,y,z\in A$.
\end{defi}

Recall that $(A,\circ)$ is called a \textbf{Lie-admissible
algebra}, where $A$ is a vector space with a bilinear operation
$\circ: A\otimes A\rightarrow A$, if the operation $[-,-]:A\otimes
A\rightarrow A$ given by Eq.~(\ref{eq:defi:anti-pre-Lie
algebras3}) makes $(A,[-,-])$ a Lie algebra. In this case,
$(A,[-,-])$ is called the \textbf{sub-adjacent Lie algebra} of
$(A,\circ)$ which is denoted by $(\frak g(A),[-,-])$ and
$(A,\circ)$ is called a {\bf compatible} (Lie-admissible) algebra
structure on the Lie algebra $(\frak g(A),[-,-])$.

We have the following characterization of anti-pre-Lie algebras.

\begin{pro}\label{pro:anti-pre-Lie is Lie-admissible}
Let $A$ be a vector space with a bilinear operation $\circ:
A\otimes A\rightarrow A$. Then the following statements are
equivalent:
\begin{enumerate} \item $(A,\circ)$ is an anti-pre-Lie algebra.

\item For $(A,\circ)$,  Eq.~(\ref{eq:defi:anti-pre-Lie algebras1})
and the following equation hold:
\begin{equation}\label{eq:defi:anti-pre-Lie algebras5}
x\circ[y,z]+y\circ[z,x]+z\circ[x,y]=0, \forall x,y,z\in A.
\end{equation}

\item $(A,\circ)$ is a Lie-admissible algebra such that
$(-\mathcal{L}_{\circ},A)$ is a representation of the sub-adjacent
Lie algebra $(\frak g (A),[-,-])$, where
$\mathcal{L}_{\circ}:\frak g(A)\rightarrow \mathrm{End}(A)$ is a
linear map defined by $\mathcal{L}_{\circ}(x)y=x\circ y,\forall
x,y\in A$.
\end{enumerate}
\end{pro}
\begin{proof}
$(1)\Longleftrightarrow (2)$. Let $x,y,z\in A$. Suppose that
Eq.~(\ref{eq:defi:anti-pre-Lie algebras1}) holds. Then we have
\begin{eqnarray}
&&x\circ[y,z]+y\circ[z,x]+z\circ[x,y]\nonumber\\
&&\stackrel{(\ref{eq:defi:anti-pre-Lie algebras3})}{=}
x\circ(y\circ z)-x\circ(z\circ y)+y\circ(z\circ x)-y\circ(x\circ z)+z\circ(x\circ y)-z\circ(y\circ x)\nonumber\\
&&\stackrel{(\ref{eq:defi:anti-pre-Lie algebras1})}{=}[y,x]\circ
z+[z,y]\circ x+[x,z]\circ y.\label{eq:equal}
\end{eqnarray}
Thus Eq.~(\ref{eq:defi:anti-pre-Lie algebras2}) holds if and only
if Eq.~(\ref{eq:defi:anti-pre-Lie algebras5}) holds.

$(1)\Longleftrightarrow (3)$. Let $x,y,z\in A$. Suppose that
$(A,\circ)$ is an anti-pre-Lie algebra. Then
Eqs.~(\ref{eq:defi:anti-pre-Lie algebras2}) and
(\ref{eq:defi:anti-pre-Lie algebras5}) hold. Hence we have
\begin{eqnarray*}
&&[[x,y],z]+[[y,z],x]+[[z,x],y]\\
&&=([x,y]\circ z+[y,z]\circ x+[z,x]\circ y)-(z\circ [x,y]+x\circ
[y,z]+y\circ [z,x])=0.
\end{eqnarray*}
So $(A,[-,-])$ is a Lie algebra and thus $(A,\circ)$ is
Lie-admissible. Moreover, by Eq.~(\ref{eq:defi:anti-pre-Lie
algebras1}), we have
\begin{equation}\label{eq:repn}\mathcal{L}_{\circ}(x)\mathcal{L}_{\circ}(y)-\mathcal{L}_{\circ}(y)\mathcal{L}_{\circ}(x)=-\mathcal{L}_{\circ}([x,y]),\;\;\forall x,y\in A.\end{equation}
Hence $(-\mathcal{L}_{\circ},A)$ is a representation of the Lie
algebra $(\frak g(A) ,[-,-])$. Conversely, suppose that $(A,\circ)$ is a
Lie-admissible algebra such that $(-\mathcal{L}_{\circ},A)$ is a
representation of the Lie algebra $(\frak g(A),[-,-])$. Note that
in this case, Eq.~(\ref{eq:repn}) holds, and thus equivalently,
Eq.~(\ref{eq:defi:anti-pre-Lie algebras1}) holds. Moreover,
Eq.~(\ref{eq:equal}) holds, too. Furthermore, since
 $(A,\circ)$ is a
Lie-admissible algebra, we have
$$[x,y]\circ z+[y,z]\circ x+[z,x]\circ y=z\circ [x,y]+x\circ
[y,z]+y\circ [z,x],\;\;\forall x,y,z\in A.$$ Therefore by
Eq.~(\ref{eq:equal}), Eqs.~(\ref{eq:defi:anti-pre-Lie algebras2})
and (\ref{eq:defi:anti-pre-Lie algebras5}) hold and thus
 $(A,\circ)$ is an anti-pre-Lie algebra.
\end{proof}

\begin{rmk}
Recall (\cite{Bur}) that $(A,\star)$ is called a \textbf{pre-Lie
algebra}, where $A$ is a vector space with a bilinear operation
$\star: A\otimes A\rightarrow A$, if
\begin{equation}\label{eq:defi:pre-Lie algebras}
(x\star y)\star z-x\star(y\star z)=(y\star x)\star z-y\star(x\star
z),\;\;\forall x,y,z\in A.
\end{equation}
That is, a pre-Lie algebra $(A,\star)$ is a Lie-admissible algebra
such that $(\mathcal{L}_{\star},A)$ is a representation of the
sub-adjacent Lie algebra $(\frak g(A),[-,-])$. Therefore the
notion of anti-pre-Lie algebras is justified since an anti-pre-Lie
algebra $(A,\circ)$ is a Lie-admissible algebra such that
$(-\mathcal{L}_{\circ},A)$ is a representation of the sub-adjacent
Lie algebra $(\frak g(A),[-,-])$. Note that for a pre-Lie algebra
$(A,\star)$, Eq.~(\ref{eq:defi:pre-Lie algebras}) is enough to
make the Jacobi identity for the sub-adjacent Lie algebra $(\frak
g(A),[-,-])$ hold automatically, whereas for an anti-pre-Lie
algebra $(A,\circ)$, only Eq.~(\ref{eq:defi:anti-pre-Lie
algebras1}) cannot do so and hence the additional
Eq.~(\ref{eq:defi:anti-pre-Lie algebras2}) or
Eq.~(\ref{eq:defi:anti-pre-Lie algebras5}) is needed.
\end{rmk}

\begin{rmk} Let $A$ be a vector space with a bilinear operation $\circ:
A\otimes A\rightarrow A$. $(A,\circ)$ is called {\bf two-sided
Alia} (\cite{Dzh09}) if Eqs.~(\ref{eq:defi:anti-pre-Lie
algebras2}) and (\ref{eq:defi:anti-pre-Lie algebras5}) hold.
Two-sided Alia algebras are Lie-admissible algebras and they play
an important role in the study of non-associative algebras
satisfying certain skew-symmetric identities of degree 3 (see
\cite{Dzh09,Dzh092} for more details). An anti-pre-Lie algebra is
a two-sided Alia algebra satisfying Eq.~(\ref{eq:defi:anti-pre-Lie
algebras1}).

\end{rmk}


\begin{ex} Let $A$ be a vector space with a bilinear operation $\circ:
A\otimes A\rightarrow A$. Define the
\textbf{anti-associator} as
\begin{equation}\label{eq:defi:anti-associator}
(x,y,z)_{aa}=(x\circ y)\circ z+x\circ(y\circ z),\;\; \forall x,y,z\in A.
\end{equation}
Then Eq.~(\ref{eq:defi:anti-pre-Lie
algebras1}) holds if and only if
$$(x,y,z)_{aa}=(y,x,z)_{aa},\;\;\forall x,y,z\in A,$$
that is, the anti-associator is \textbf{left-symmetric} in the
first two variables $x,y$. On the other hand, $(A,\circ)$ is
called \textbf{anti-associative} (\cite{Oku}) if for all $x,y,z\in
A$,  $(x,y,z)_{aa}=0$. Obviously, an anti-associative algebra is
an anti-pre-Lie algebra if and only if it is a Lie-admissible
algebra, that is, Eq.~(\ref{eq:defi:anti-pre-Lie algebras2})
holds. In particular, 
$(A,\circ)$ is both associative
and anti-associative if and only if $(A,\circ)$ is a 2-step
nilpotent associative algebra, that is,
$$x\circ(y\circ z)=(x\circ y)\circ z=0,\;\;\forall x,y,z\in A.$$
In this case, $(A,\circ)$ is an anti-pre-Lie algebra.
\end{ex}


\begin{pro}\label{pro:comm}
 Let $A$ be a vector space with a bilinear operation $\circ:
A\otimes A\rightarrow A$. Suppose that the operation is
commutative, that is,
$$x\circ y=y\circ x,\;\;\forall x,y\in A.$$
Then  $(A,\circ)$ is an anti-pre-Lie algebra if and only if
$(A,\circ)$ is associative.
\end{pro}

\begin{proof} Obviously
Eq.~(\ref{eq:defi:anti-pre-Lie algebras2}) holds. Moreover,
$$x\circ(y\circ z)-y\circ(x\circ z)-[y,x]\circ z=(y\circ z)\circ x-y\circ(z\circ x),\forall x,y,z\in A.$$
Therefore Eq.~(\ref{eq:defi:anti-pre-Lie algebras1}) holds if and
only if $(A,\circ)$ is associative. Hence the conclusion follows.
\end{proof}

\begin{pro}\label{pro:anti-pre-Lie algebras from linear functions}
Let $A$ be a vector space of dimension $n\geq 2$, and
$f,g:A\rightarrow\mathbb{F}$ be two linear functions. Define a
bilinear operation $\circ:A\otimes A\rightarrow A$ by
\begin{equation}\label{eq:anti-pre-Lie algebras from linear functions}
x\circ y=f(y)x+g(x)y, \forall x,y\in A.
\end{equation}
Then $(A,\circ)$ is an anti-pre-Lie algebra if and only if $f=0$
or $g=2f$.
\end{pro}

\begin{proof}
For all $x,y,z\in A$, we have
\begin{eqnarray*}
&&[y,x]\circ z
=f(x)f(z)y-g(x)f(z)y+f(x)g(y)z-f(y)g(x)z+g(y)f(z)x-f(y)f(z)x.
\end{eqnarray*}
Then Eq.~(\ref{eq:defi:anti-pre-Lie algebras2}) holds.
Moreover, 
Eq.~(\ref{eq:defi:anti-pre-Lie algebras1}) holds if and only
if
\begin{equation}\label{eq:anti-pre-Lie algebra from linear function1}
2f(x)f(z)y-g(x)f(z)y+f(x)g(y)z-f(y)g(x)z+g(y)f(z)x-2f(y)f(z)x=0,\;\;\forall
x,y,z\in A.
\end{equation}
\begin{enumerate}
\item Suppose that $\mathrm{dim}A\geq 3$. Then for any $x,z\in A$, there is a vector $y\in A$ that is linearly independent of $x$ and $z$. Hence by Eq.~(\ref{eq:anti-pre-Lie algebra from
linear function1}), we have $(2f-g)(x)f(z)=0, \forall x,z\in A$. Thus $f=0$ or $g=2f$.
\item Suppose that $\mathrm{dim}A=2$. Let $\{ e_1,e_2\}$ be a basis of $A$. Then Eq.~(\ref{eq:anti-pre-Lie algebra from
linear function1}) holds if and only if 
\begin{equation}\label{eq:A}
(2f(e_{1})-g(e_{1}))f(e_{1})=0,
\end{equation}
\begin{equation}\label{eq:B}
(2f(e_{2})-g(e_{2}))f(e_{2})=0,
\end{equation}
\begin{equation}\label{eq:C}
2f(e_{1})(g(e_{2})-f(e_{2}))-f(e_{2})g(e_{1})=0,
\end{equation}
\begin{equation}\label{eq:D}
2f(e_{2})(g(e_{1})-f(e_{1}))-f(e_{1})g(e_{2})=0.
\end{equation}
\begin{itemize}
\item[(i)] If $f(e_{1})\neq 0$, then by Eqs.~(\ref{eq:A}) and (\ref{eq:D}), we have $g(e_{1})=2f(e_{1}), g(e_{2})=2f(e_{2})$. Thus $g=2f$.
\item[(ii)] If $f(e_{1})=0$, then by Eq.~(\ref{eq:D}), we have $f(e_{2})g(e_{1})=0$.
\begin{itemize}
\item[(ii)-(a)] If $f(e_{2})=0$, then $f=0$.
\item[(ii)-(b)] If $f(e_{2})\neq 0$, then $g(e_{1})=0$. By Eq.~(\ref{eq:B}) we have $g(e_{2})=2f(e_{2})$. Thus $g=2f$.
\end{itemize}
\end{itemize}
\end{enumerate}
It is straightforward that Eq.~(\ref{eq:anti-pre-Lie
algebra from linear function1}) holds when $f=0$ or $g=2f$. Hence
the conclusion holds.
\end{proof}

\begin{cor}\label{cor:class}
With the conditions in Proposition~\ref{pro:anti-pre-Lie algebras
from linear functions}, we have the following results.
\begin{enumerate}
\item \label{it:f1}If $f=0$ and $g\ne 0$, then there exists a
basis $\{e_1,\cdots, e_n\}$ in $A$ such that the non-zero products
are given by
$$e_1\circ e_i=e_i, i=1,\cdots, n.$$
\item  \label{it:f2}If $f\ne 0$ and $g=2f$, then there exists a
basis $\{e_1,\cdots, e_n\}$ in $A$ such that the non-zero products
are given by
$$e_1\circ e_1=3e_1,\;\;e_1\circ e_i=2e_i,\;\;e_i\circ e_1=e_{i},\;\;i=2,\cdots, n.$$
\item  \label{it:f3} If $f=g=0$, then $A$ is trivial, that is, all
products are zero.
\end{enumerate}

\end{cor}

\begin{proof}
For any non-zero linear function $f:A\rightarrow \mathbb F$, since
$A={\rm Ker} f\oplus f(A)={\rm Ker} f\oplus \mathbb F$, there
exists a basis  $\{e_1,\cdots, e_n\}$ of $A$ such that $f(e_1)=1,
f(e_i)=0, i=2,\cdots,n$. Hence Case (\ref{it:f1}) follows (also
see \cite{Bai2004}) by taking $g(e_{1})=1, g(e_{i})=0, i=2,
\cdots, n$ and Case (\ref{it:f2}) follows by taking $f(e_{1})=1,
f(e_{i})=0, i=2, \cdots, n$. Case (\ref{it:f3}) is obvious.
\end{proof}

\begin{rmk}
In fact, by \cite{Bai2004}, the assumption defined by
Eq.~(\ref{eq:anti-pre-Lie algebras from linear functions}) is
equivalent to the fact that for any $x,y\in A$, $x\circ y$ is
still in the subspace spanned by $x$ and $y$. Moreover,
Eq.~(\ref{eq:anti-pre-Lie algebras from linear functions}) defines
a pre-Lie algebra if and only if $f=0$ or $g=0$. So when $f=0$ or
equivalently, in Cases (\ref{it:f1}) and (\ref{it:f3}) in
Corollary~\ref{cor:class},
$(A,\circ)$ is both a pre-Lie algebra and an anti-pre-Lie algebra.
Furthermore, in these cases, $(A,\circ)$ is associative.
\end{rmk}




At the end of this subsection, we consider the classification of
complex anti-pre-Lie algebras in dimension 2. By
Proposition~\ref{pro:comm}, the commutative anti-pre-Lie algebras
are commutative associative algebras whose classification is
known (cf.~\cite{Bai2001} or~\cite{Burde1998}). So we give the classification of
2-dimensional complex non-commutative anti-pre-Lie algebras as
follows.

\begin{pro}\label{thm:classification}
Let $(A,\circ)$ be a 2-dimensional non-commutative anti-pre-Lie algebra over the complex field
$\mathbb{C}$ with a basis $\{e_{1},e_{2}\}$. Then $(A,\circ)$ is isomorphic to one of the following mutually non-isomorphic cases:
\begin{enumerate}

\item[$(A1)_{\mbox{\ }}$] $e_{1}\circ e_{1}=-e_{2}, e_{1}\circ
e_{2}=0, e_{2}\circ e_{1}=-e_{1}, e_{2}\circ e_{2}=0;$

\item[$(A2)_\lambda$] ($\lambda\in\mathbb{C}$)   $e_{1}\circ
e_{1}=0, e_{1}\circ e_{2}=0, e_{2}\circ e_{1}=-e_{1}, e_{2}\circ
e_{2}=\lambda e_{2};$

\item[$(A3)_{\mbox{\ }}$] $e_{1}\circ e_{1}=0, e_{1}\circ e_{2}=0,
e_{2}\circ e_{1}=-e_{1}, e_{2}\circ e_{2}=e_{1}-e_{2};$

\item[$(A4)_\lambda$] ($\lambda\ne -1$)  $e_{1}\circ e_{1}=0,
e_{1}\circ e_{2}=(\lambda+1)e_{1}, e_{2}\circ e_{1}=\lambda e_{1},
e_{2}\circ e_{2}=(\lambda-1)e_{2}$;

\item[$(A5)_{\mbox{\ }}$] $e_{1}\circ e_{1}=0, e_{1}\circ
e_{2}=-e_{1}, e_{2}\circ e_{1}=-2e_{1}, e_{2}\circ
e_{2}=e_{1}-3e_{2}.$

\end{enumerate}
\end{pro}

\begin{proof}
Due to the classification of 2-dimensional complex non-abelian Lie
algebras, we assume that $[e_1,e_2]=e_1$. Set $e_{i}\circ
e_{j}=a_{ij}e_{1}+b_{ij}e_{2}$, where $a_{ij},b_{ij}\in\mathbb{C}$
and $i,j\in\{ 1,2\}$. Then we have
$$a_{12}-a_{21}=1, b_{12}=b_{21}.$$
Moreover, in this case, Eq.~(\ref{eq:defi:anti-pre-Lie algebras2}) holds
automatically, and Eq.~(\ref{eq:defi:anti-pre-Lie algebras1})
holds if and only if the following equations hold:
\begin{equation}\label{eq:thm:classification1}
e_{1}\circ (e_{2}\circ e_{1})-e_{2}\circ(e_{1}\circ e_{1})=-e_{1}\circ e_{1},
\end{equation}
\begin{equation}\label{eq:thm:classification2}
e_{1}\circ (e_{2}\circ e_{2})-e_{2}\circ(e_{1}\circ e_{2})=-e_{1}\circ e_{2}.
\end{equation}

\begin{enumerate} \item If $b_{12}=b_{21}\neq 0$, then replacing $e_{1}$ by $\dfrac{1}{b_{12}}e_{1}$, we can assume that $b_{12}=b_{21}=1$. Thus
by Eqs.~(\ref{eq:thm:classification1}) and (\ref{eq:thm:classification2}), we have
\begin{eqnarray*}
&&a_{11}+a_{21}-a_{22}b_{11}+1=0,\;\;
a_{11}-a_{21}b_{11}-b_{11}+b_{11}b_{22}-1=0,\\
&&a_{11}a_{22}-a_{21}^{2}+a_{21}b_{22}-a_{22}+b_{22}+1=0,\;\;
a_{21}-a_{22}b_{11} =0.
\end{eqnarray*}
Therefore we have
$$a_{11}=-1, b_{11}\neq 0, a_{21}=-\dfrac{b_{11}+1}{b_{11}}, a_{22}=-\dfrac{b_{11}+1}{b^{2}_{11}}, b_{22}=\dfrac{1}{b_{11}}.$$
Let $e_{1}\rightarrow\dfrac{1}{\sqrt{-b_{11}}}e_{1}, e_{2}\rightarrow -\dfrac{1}{b_{11}}e_{1}+e_{2}$, we get $(A1)$.\\

\item If $b_{12}=b_{21}=0$, 
then by Eqs.~(\ref{eq:thm:classification1}) and (\ref{eq:thm:classification2}), we have
\begin{eqnarray*}
&&a_{11}-a_{22}b_{11}=0,\;\; a_{21}b_{11}-b_{11}b_{22}+b_{11}=0,\\
&&a_{11}a_{22}+a_{21}b_{22}+b_{22}+(1-a_{21})(a_{21}+1)=0,
a_{22}b_{11}=0.
\end{eqnarray*}
Therefore we have
$$a_{11}=0, a_{22}b_{11}=0, b_{11}(a_{21}-b_{22}+1)=0, (a_{21}+1)(a_{21}-b_{22}-1)=0. $$
\begin{itemize}
\item[(i)] If $b_{11}\neq 0$, then $a_{22}=0, a_{21}=-1, b_{22}=0.$
Let $e_{1}\rightarrow \dfrac{\mathrm{1}}{\sqrt{-b_{11}}}e_{1}, e_{2}\rightarrow -e_{2}$, we still get $(A1)$.

\item[(ii)] If $b_{11}=0$, then the only constraint left is
$$(a_{21}+1)(a_{21}-b_{22}-1)=0.$$
\begin{itemize}
\item[(ii)-(a)] If $a_{21}=-1,$ then $a_{22},b_{22}$ are arbitrary.
\begin{itemize}
\item Let $a_{21}=-1, b_{22}=0.$ Let $e_1\rightarrow e_1, e_{2}\rightarrow a_{22}e_{1}+e_{2}$, we get
$(A2)_0$.

\item  Let $a_{21}=-1, b_{22}\neq 0, -1$. Let
$e_1\rightarrow e_1, e_{2}\rightarrow \dfrac{a_{22}}{b_{22}+1}e_{1}+e_{2}$, we get
$(A2)_\lambda$ with $\lambda\neq 0, -1$ by setting
$\lambda=b_{22}$.

\item Let $a_{21}=-1, b_{22}=-1, a_{22}=0$. Then we get
$(A2)_{-1}$.

\item Let $a_{21}=-1, b_{22}=-1, a_{22}\neq 0$. Let
$e_{1}\rightarrow a_{22}e_{1}, e_{2}\rightarrow e_{2}$, we get
$(A3)$.
\end{itemize}
\item[(ii)-(b)] If $a_{21}\neq -1$, then $a_{22}$ is arbitrary and $b_{22}=a_{21}-1$.
\begin{itemize}

\item Let $a_{21}\neq-1, a_{21}\neq-2$. Let $e_1\rightarrow e_1, e_{2}\rightarrow -\dfrac{a_{22}}{a_{21}+2}e_{1}+e_{2}$, we get
$(A4)_\lambda$ with $\lambda\neq -1, -2$ by setting
$a_{21}=\lambda$.

\item Let $a_{21}=-2, a_{22}=0$. Then we get $(A4)_{-2}$.

\item Let $a_{21}=-2, a_{22}\neq 0$. Let $e_{1}\rightarrow
a_{22}e_{1}, e_{2}\rightarrow e_{2}$, we get
$(A5)$.
\end{itemize}
\end{itemize}
\end{itemize}

\end{enumerate}
It is straightforward to show that they are anti-pre-Lie algebras and moreover,
they are not isomorphic mutually.  Therefore the conclusion follows.
\end{proof}

The following conclusion is obtained directly. Note that an
anti-pre-Lie algebra is called {\bf simple} if there is not an
ideal except for zero and itself.

\begin{cor} With the notations in Proposition~\ref{thm:classification}, among the 2-dimensional complex non-commutative anti-pre-Lie
algebras, we have the following results.
\begin{enumerate}
\item Associative algebras: $(A2)_{-1}$.
 \item Anti-associative algebras: None.  Furthermore, since any 2-dimensional anti-associative algebra is automatically Lie-admissible and thus an anti-pre-Lie algebra, there is not a complex non-commutative anti-associative
algebra in dimension 2. \item Pre-Lie algebras: $(A2)_\lambda$
($\lambda\in \mathbb C$) and $(A3)$.
  \item In Corollary~\ref{cor:class} with $\dim
A=2$, Case (\ref{it:f1}) is isomorphic to $(A2)_{-1}$ and Case
(\ref{it:f2}) is isomorphic to $(A4)_{-2}$. \item There is not a
complex simple anti-pre-Lie algebra in dimension 2.
\end{enumerate}
\end{cor}



\subsection{Anti-pre-Lie algebras and anti-$\mathcal{O}$-operators on Lie algebras}
We introduce the notion of anti-$\mathcal{O}$-operators to
interpret anti-pre-Lie algebras. There is an anti-pre-Lie algebra
structure on the representation space inducing from a strong
anti-$\mathcal O$-operator on a Lie algebra and in particular, the
existence of an invertible anti-$\mathcal O$-operator gives an
equivalent condition for a Lie algebra having a compatible
anti-pre-Lie algebra structure.

Let $(\frak g,[-,-])$ be a Lie algebra. Let ${\rm ad}:\frak
g\rightarrow {\rm End}(\frak g)$ be the linear map defined by ${\rm
ad}(x)(y)=[x,y]$ for all $x,y\in \frak g$. Then $(\mathrm{ad},
\frak g)$ is a representation of $(\frak g,[-,-])$, called the
\textbf{adjoint representation} of $(\frak g,[-,-])$.

\begin{defi}\label{defi:anti O-operators}
Let $(\frak g,[-,-])$ be a Lie algebra and $(\rho,V)$ be a
representation. A linear map $T:V\rightarrow \frak g$ is called an
\textbf{anti-$\mathcal O$-operator} of $(\frak g,[-,-])$
associated to $(\rho,V)$ if $T$ satisfies
\begin{equation}\label{eq:defi:anti O-operators1}
[T(u),T(v)]=T(\rho(T(v))u-\rho(T(u))v), \forall u,v\in V.
\end{equation}
An anti-$\mathcal O$-operator $T$ is called \textbf{strong} if $T$
satisfies
\begin{equation}\label{eq:defi:anti O-operators2}
\rho([T(u),T(v)])w+\rho([T(v),T(w)])u+\rho([T(w),T(u)])v=0,
\forall u,v,w\in V.
\end{equation}
In particular, an anti-$\mathcal O$-operator $R$ of $(\frak
g,[-,-])$ associated to the adjoint representation $({\rm
ad},\frak g)$ is called an {\bf anti-Rota-Baxter operator}, that
is, $R:\frak g\rightarrow \frak g$ is a linear map satisfying
\begin{equation}\label{eq:defi:anti Rota-Baxter operator}
[R(x),R(y)]=R([R(y),x]+[y,R(x)]),\forall x,y\in \frak g.
\end{equation}
An anti-Rota-Baxter operator $R$ is called {\bf strong} if $R$
satisfies
\begin{equation}\label{eq:RB-strong}
 [[R(x),R(y)],z]+[[R(y),R(z)],x]+[[R(z),R(x)],y]=0,\;\;\forall
 x,y,z\in \frak g.
\end{equation}
\end{defi}

\begin{rmk} Let $(\frak g,[-,-])$ be a Lie algebra and $(\rho,V)$ be a
representation. On the one hand, recall that a linear map
$T:V\rightarrow \frak g$ is called an {\bf $\mathcal O$-operator}
of  $(\frak g,[-,-])$ associated to $(\rho, V)$ if $T$ satisfies
\begin{equation}\label{eq:defi: O-operators1}
[T(u),T(v)]=T(\rho(T(u))v-\rho(T(v))u), \forall u,v\in V.
\end{equation}
It was introduced in \cite{Kup} as a natural generalization of the
classical Yang-Baxter equation in a Lie algebra. Hence the notion of anti-$\mathcal O$-operators
is justified. On the other
hand, recall (\cite{Dzh, Fil}) that a linear map $D:\frak
g\rightarrow V$ is called an {\bf anti-derivation} (or an {\bf
anti-1-cocycle}) if $D$ satisfies
\begin{equation}\label{eq:defi:anti-derivation}
D([x,y])=\rho(y) (D(x))-\rho(x)(D(y)),\;\;\forall x,y\in \frak g.
\end{equation}
Obviously, an invertible linear map $T:V\rightarrow \frak g$ is an
anti-$\mathcal O$-operator if and only if $T^{-1}$ is an
anti-derivation.
\end{rmk}


\begin{pro}\label{pro:O-operator}
Let $T:V\rightarrow \frak g$ be an anti-$\mathcal O$-operator of a
Lie algebra $(\frak g,[-,-]_{\frak g})$ associated to a
representation $(\rho,V)$. Define a bilinear operation $\circ: V\otimes V\rightarrow V$ by
\begin{equation}\label{eq:pro:O-operator}
u\circ v=-\rho(T(u))v, \forall u,v\in V.
\end{equation}
Then $(V,\circ)$ satisfies Eq.~(\ref{eq:defi:anti-pre-Lie
algebras1}). Moreover, $(V,\circ)$ is Lie-admissible such that
$(V,\circ)$ is an anti-pre-Lie algebra if and only if $T$ is
strong.
In this case, $T$ is a homomorphism of Lie algebras from the
sub-adjacent Lie algebra $(\frak g(V),[-,-])$ to $(\frak
g,[-,-]_{\frak g})$. Furthermore, there is an induced anti-pre-Lie
algebra structure on $T(V)=\{T(u)|u\in V\}\subset \frak g$ given
by
\begin{equation}\label{eq:pro:O-operator2}
T(u)\circ_{\frak g}T(v)=T(u\circ v), \forall u,v\in V,
\end{equation}
and $T$ is a homomorphism of anti-pre-Lie algebras.
\end{pro}

\begin{proof}
Let $u,v,w\in V$. Then
\begin{eqnarray*}
[u,v]\circ w
&=&(\rho(T(v))u-\rho(T(u))v)\circ w=-\rho(T(\rho(T(v))u))w+\rho(T(\rho(T(u))v))w\\
&=&\rho([T(v),T(u)]_{\frak g})w=\rho(T(v))\rho(T(u))w-\rho(T(u))\rho(T(v))w\\
&=&v\circ(u\circ w)-u\circ(v\circ w).
\end{eqnarray*}
Thus Eq.~(\ref{eq:defi:anti-pre-Lie algebras1}) holds on
$(V,\circ)$. Moreover, $(V,\circ)$ is Lie-admissible if and only
if Eq.~(\ref{eq:defi:anti-pre-Lie algebras2}) holds on
$(V,\circ)$, that is,
\begin{eqnarray*}
&&[u,v]\circ w+[v,w]\circ u+[w,u]\circ v\\
&&=\rho([T(v),T(u)]_{\frak g})w+\rho([T(w),T(v)]_{\frak
g})u+\rho([T(u),T(w)]_{\frak g})v=0.
\end{eqnarray*}
It is straightforward to show  that the other results hold.
\end{proof}

\begin{cor}\label{cor:anti operator}
 Let $(\frak g,[-,-])$ be a Lie algebra and $R:\frak
g\rightarrow \frak g$ be a strong anti-Rota-Baxter operator. Then
\begin{equation}\label{eq:cons}
x\circ y=-[R(x),y],\;\;\forall x,y\in \frak g
\end{equation}
defines an anti-pre-Lie algebra. Conversely, if $R:\frak
g\rightarrow \frak g$ is a linear transformation on a Lie algebra
$(\frak g,[-,-])$ such that Eq.~(\ref{eq:cons}) defines an anti-pre-Lie
algebra, then $R$ satisfies Eq.~(\ref{eq:RB-strong}) and the
following equation:
\begin{equation}\label{eq:cons2}
[[R(x),R(y)]+R([x,R(y)]+[R(x),y]),z]=0,\;\;\forall x,y,z\in \frak g.
\end{equation}
\end{cor}

\begin{proof}
The first half part follows from Proposition~\ref{pro:O-operator}
by letting $\rho={\rm ad}$. The second half part follows from
Definition~\ref{defi:anti-pre-Lie algebras} directly.
\end{proof}


\begin{ex}\label{ex:anti operator}
Let $(\frak g,[-,-])$ be the 2-dimensional complex non-abelian Lie
algebra with a basis $\{e_1,e_2\}$ whose product is given by
$[e_{1}, e_{2}]=e_{1}$. In this case any anti-Rota-Baxter operator
is automatically strong. Let $R:\frak g\rightarrow \frak g$ be a
linear map whose corresponding matrix is given by
$\left(\begin{matrix} r_{11}&r_{12}\cr
r_{21}&r_{21}\cr\end{matrix}\right)$ under the basis
$\{e_1,e_2\}$. Then by Eq.~(\ref{eq:defi:anti Rota-Baxter
operator}), $R$ is an anti-Rota-Baxter operator on $(\frak g,[-,-])$ if
and only if
$$r_{11}(r_{11}+2r_{22})=r_{21}r_{12},\;\;
r_{21}(r_{11}+r_{22})=0.$$ By a case by case analysis, we give all
anti-Rota-Baxter operators as follows.
$$
R_{1}=\left(
\begin{array}{cc}
0 & 0\\
0 & 0\\
\end{array}
\right), R_{2}=\left(
\begin{array}{cc}
0 & a\\
0 & 0\\
\end{array}
\right), R_{3}=\left(
\begin{array}{cc}
0 & 0\\
0 & a\\
\end{array}
\right), R_{4}=\left(
\begin{array}{cc}
0 & b\\
0 & a\\
\end{array}
\right),
$$
$$
R_{5}=\left(
\begin{array}{cc}
a & 0\\
0 & -\dfrac{1}{2}a\\
\end{array}
\right), R_{6}=\left(
\begin{array}{cc}
a & b\\
0 & -\dfrac{1}{2}a\\
\end{array}
\right), R_{7}=\left(
\begin{array}{cc}
0 & 0\\
a & 0\\
\end{array}
\right), R_{8}=\left(
\begin{array}{cc}
a & -\dfrac{a^{2}}{b}\\
b & -a\\
\end{array}
\right),
$$
where the parameters $a,b\ne 0$. Furthermore, the corresponding
anti-pre Lie algebras given by Eq.~(\ref{eq:cons}) in the sense of
isomorphism are listed as follows (we use the notations in
Proposition~\ref{thm:classification} or only give the non-zero
products):
\begin{enumerate}
\item[]($R_1$): the trivial anti-pre-Lie algebra.

\item[]($R_2$): $e_{2}\circ e_{2}=e_{1}$, which is commutative
associative.

\item[]($R_3$)  \& ($R_4$): $(A2)_0$.

\item[] ($R_5$) \& ($R_6$): $(A4)_1$.

\item[] ($R_7$) \& ($R_8$): $e_{1}\circ e_{1}=e_{1}$, which is
commutative associative.

\end{enumerate}
\end{ex}

Next we consider the invertible anti-$\mathcal O$-operators.

\begin{pro}\label{thm:invertible}
An invertible anti-$\mathcal{O}$-operator of a Lie algebra is
automatically strong.
\end{pro}

\begin{proof}
Let $T:V\rightarrow \frak g$ be an anti-$\mathcal{O}$-operator of
a Lie algebra $(\frak g,[-,-]_{\frak g})$ associated to a
representation $(\rho,V)$. Define a bilinear operation $\circ: V\otimes V\rightarrow V$ by Eq.~(\ref{eq:pro:O-operator}). Then we have
$$[u,v]=u\circ v-v\circ u={\rho(T(v))u-\rho(T(u))v}=T^{-1}([T(u),T(v)]_{\frak g}),\;\;\forall u,v\in V.$$
Therefore for all $u,v,w\in V$, we have
$$[[u,v],w]=[T^{-1}([T(u),T(v)]_{\frak g}),w]=T^{-1}([[T(u),T(v)]_{\frak g},T(w)]_{\frak g}). $$
Thus $(V,[-,-])$ is a Lie algebra, and hence $(V,\circ)$ is
Lie-admissible. Then by Proposition \ref{pro:O-operator}, $T$ is
strong.
\end{proof}

\begin{cor}\label{cor:anti O-operators and anti-pre-Lie algebras}
Let $(\frak g,[-,-])$ be a Lie algebra. Then there is a compatible
anti-pre-Lie algebra structure on $\frak g$ if and only if there
exists an invertible anti-$\mathcal{O}$-operator of $(\frak
g,[-,-])$.
\end{cor}
\begin{proof}
Suppose that $(\frak g,\circ)$ is a compatible anti-pre-Lie
algebra structure on $(\frak g,[-,-])$. Then
$$[x,y]=x\circ y-y\circ x=-\mathcal{L}_{\circ}(y)x-(-\mathcal{L}_{\circ}(x))y, \forall x,y\in \frak g. $$
Thus $T=\mathrm{id}$ is an invertible anti-$\mathcal O$-operator
of $(\frak g,[-,-])$ associated to $(-\mathcal{L}_{\circ},\frak
g)$.

Conversely, suppose that $T:V\rightarrow \frak g$ is an invertible
anti-$\mathcal O$-operator of $(\frak g,[-,-])$ associated to
$(\rho, V)$. Then by Proposition~\ref{thm:invertible}, $T$ is
strong. And by Proposition~\ref{pro:O-operator}, there is an
anti-pre-Lie algebra structure on $V$ given by
Eq.~(\ref{eq:pro:O-operator}). Moreover, since $T$ is invertible,
there is an anti-pre-Lie algebra structure on the underlying
vector space of $\frak g$ given by Eq.~(\ref{eq:pro:O-operator2}).
Explicitly, for any $x,y\in \frak g$, there exist $u,v\in V$ such
that $x=T(u), y=T(v)$. Hence
\begin{equation}
x\circ_{\frak g} y=T(u)\circ_{\frak g} T(v)=T(u\circ
v)=-T(\rho(x)T^{-1}(y)), \forall x,y\in \frak g.
\end{equation}
Furthermore, by Eq.~(\ref{eq:defi:anti O-operators1}), we have
\begin{eqnarray*}
[x,y]=[T(u),T(v)]=T(\rho(T(v))u-\rho(T(u))v)=T(\rho(y)T^{-1}(x)-\rho(x)T^{-1}(y))
=x\circ_{\frak g} y-y\circ_{\frak g} x.
\end{eqnarray*}
Thus $(\frak g,\circ_{\frak g})$ is an anti-pre-Lie algebra whose
sub-adjacent Lie algebra is $(\frak g,[-,-])$.
\end{proof}

\subsection{Anti-pre-Lie algebras and commutative 2-cocycles on Lie algebras}
As a direct consequence of the study 
in the previous subsection, we illustrate that there are
anti-pre-Lie algebras obtained from nondegenerate commutative
2-cocycles on Lie algebras. Conversely, there is a natural
construction of nondegenerate commutative 2-cocycles on
semi-direct product Lie algebras induced from anti-pre-Lie
algebras. There is also a construction as well as the
classification of anti-pre-Lie algebras from symmetric bilinear
forms in which the bilinear forms are commutative 2-cocycles on
the sub-adjacent Lie algebras.

Let $(\rho,V)$ be a representation of a Lie algebra $(\frak
g,[-,-])$.  Then $(\rho^{*},V^{*})$  is also a representation of
$(\frak g,[-,-])$, where the linear map $\rho^{*}:\frak
g\rightarrow\mathrm{End}(V^{*})$ is defined as
$$
\langle\rho^{*}(x)u^{*},v\rangle=-\langle
u^{*},\rho(x)v\rangle,\forall x\in \frak g, u^{*}\in V^{*},v\in V.
$$
In particular,
  $(\mathrm{ad}^{*},\frak g^{*})$ is a representation of $(\frak g,[-,-])$.

\begin{thm}\label{thm:commutative 2-cocycles and anti-pre-Lie algebras}
Let $\mathcal{B}$ be a nondegenerate commutative 2-cocycle on a
Lie algebra $(\frak g,[-,-])$. Then there exists a compatible
anti-pre-Lie algebra structure $\circ$ on $(\frak g,[-,-])$ given
by
\begin{equation}\label{eq:thm:commutative 2-cocycles and anti-pre-Lie algebras}
\mathcal{B}(x\circ y,z)=\mathcal{B}(y,[x,z]), \;\;\forall x,y,z\in
\frak g.
\end{equation}
\end{thm}

\begin{proof}
Define a linear map $T:\frak g\rightarrow \frak g^{*}$ by
\begin{equation}\label{eq:thm:commutative 2-cocycles and anti-pre-Lie algebras2}
\langle T(x),y\rangle=\mathcal{B}(x,y),\;\; \forall x,y\in \frak
g.
\end{equation}
Then $T$ is invertible by the nondegeneracy of $\mathcal{B}$.
Moreover, for any $a^*,b^*\in \frak g^*$, there exist $x,y\in
\frak g$ such that $a^*=T(x), b^*=T(y)$. Then for any $z\in \frak
g$, we have
\begin{eqnarray*}
&&\mathcal{B}([x,y],z)=\langle T([x,y]),z\rangle=\langle
T([T^{-1}(a^{*}),T^{-1}(b^{*})]), z\rangle,\\
&&\mathcal{B}([y,z],x)=\langle
T(x),[y,z]\rangle=-\langle\mathrm{ad}^{*}(y)T(x),z\rangle=-\langle\mathrm{ad}^{*}(T^{-1}(b^{*}))a^{*},z\rangle,\\
&&\mathcal{B}([z,x],y)=\langle
T(y),[z,x]\rangle=\langle\mathrm{ad}^{*}(x)T(y),z\rangle=\langle\mathrm{ad}^{*}(T^{-1}(a^{*}))b^{*},z\rangle.
\end{eqnarray*}
Since $\mathcal{B}$ is a commutative 2-cocycle, we have
$$[T^{-1}(a^{*}),T^{-1}(b^{*})]=T^{-1}(\mathrm{ad}^{*}(T^{-1}(b^{*}))a^{*}-\mathrm{ad}^{*}(T^{-1}(a^{*}))b^{*}).$$
Thus $T^{-1}$ is an anti-$\mathcal{O}$-operator of $(\frak
g,[-,-])$ associated to $(\mathrm{ad}^{*},\frak g^{*})$. By
Corollary \ref{cor:anti O-operators and anti-pre-Lie algebras},
there is a compatible anti-pre-Lie algebra structure $\circ$ on
$(\frak g,[-,-])$ given by
\begin{equation*}\label{eq:compatible anti-pre-Lie algebra structure}
x\circ y=-T^{-1}(\mathrm{ad}^{*}(x)T(y)),\;\; \forall x,y\in \frak
g,
\end{equation*}
which gives
$$\mathcal{B}(x\circ y,z)=\langle T(x\circ y),z\rangle=-\langle\mathrm{ad}^{*}(x)T(y),z\rangle=\langle T(y), [x,z]\rangle=\mathcal{B}(y,[x,z]),\;\;\forall x,y,z\in \frak g.$$
Hence the conclusion holds.
\end{proof}



\begin{ex} There are two approaches by applying Theorem \ref{thm:commutative 2-cocycles and anti-pre-Lie
algebras}:
one is to consider all nondegenerate commutative 2-cocycles on a
fixed Lie algebra and another is to consider all Lie algebras
admitting a fixed nondegenerate symmetric bilinear form such that it is a commutative 2-cocycle. As an
example, we apply and illustrate the second approach in dimension
2. Let $\mathcal{B}$ be a nondegenerate symmetric bilinear form on
the complex 2-dimensional non-abelian Lie algebra
$(\mathfrak{g},[-,-])$. Note that in this case any symmetric
bilinear form is a commutative 2-cocycle. Hence there exists a
basis $\{e_{1}, e_{2}\}$ of $\frak g$ such that the non-zero values of
$\mathcal{B}$ are given by $\mathcal{B}(e_{1},
e_{1})=\mathcal{B}(e_{2}, e_{2})=1$. Assume that $[e_{1},
e_{2}]=ae_{1}+be_{2}$, where $a\neq 0$ or $b\neq 0$. Then by
Theorem \ref{thm:commutative 2-cocycles and anti-pre-Lie
algebras}, we have the following compatible anti-pre Lie algebras:
$$e_{1}\circ e_{1}=ae_{2}, e_{1}\circ e_{2}=be_{2}, e_{2}\circ e_{1}=-ae_{1}, e_{2}\circ e_{2}=-be_{1}.$$
By a similar proof as the one for
Proposition~\ref{thm:classification}, they are isomorphic to
$(A1)$ or $(A4)_0$ in Proposition~\ref{thm:classification}.
\end{ex}

\begin{ex}\label{ex:anti-pre-Lie algebra from sl(2)}
Let $(\frak g,[-,-])$ be the 3-dimensional simple Lie algebra
$\frak s\frak l(2,\mathbb{F})$ with a basis $\{
x=\left(\begin{matrix} 0&1\cr
0&0\cr\end{matrix}\right),h=\left(\begin{matrix} 1&0\cr
0&-1\cr\end{matrix}\right),y=\left(\begin{matrix} 0&0\cr
1&0\cr\end{matrix}\right)\}$ whose products are given by
$$[h,x]=2x,[h,y]=-2y,[x,y]=h.$$
It is straightforward to check that the bilinear form $\mathcal B$
whose non-zero values are given by
$$\mathcal{B}(x,y)=\mathcal{B}(y,x)=-1,\mathcal{B}(h,h)=4$$
is a nondegenerate commutative 2-cocycle on $(\frak g,[-,-])$.
Thus by Theorem \ref{thm:commutative 2-cocycles and anti-pre-Lie
algebras}, there is a compatible anti-pre-Lie algebra structure on
$(\frak g,[-,-])$ given by the following non-zero products:
$$x\circ h=-4x, x\circ y=\dfrac{1}{2}h, h\circ x=-2x, h\circ y=2y, y\circ x=-\dfrac{1}{2}h, y\circ h=4y.$$
Since any ideal of an anti-pre-Lie algebra is still an ideal of
the sub-adjacent Lie algebra, any compatible anti-pre-Lie algebra
structure on a simple Lie algebra is simple. Hence the above
anti-pre-Lie algebra is simple.
\end{ex}

\begin{defi}\label{defi:quadratic}
A bilinear form $\mathcal B$ on an anti-pre-Lie algebra
$(A,\circ)$ is called \textbf{invariant}  if
Eq.~(\ref{eq:thm:commutative
2-cocycles and anti-pre-Lie algebras}) holds. %

\end{defi}

\begin{cor}\label{cor:corres}
Any symmetric invariant bilinear form on an anti-pre-Lie algebra
$(A,\circ)$ is a commutative 2-cocycle on the sub-adjacent Lie
algebra $(\frak g(A),[-,-])$. Conversely, a nondegenerate
commutative 2-cocycle on a Lie algebra $(\frak g,[-,-])$ is
invariant on the compatible anti-pre-Lie algebra given by
Eq.~(\ref{eq:thm:commutative 2-cocycles and anti-pre-Lie
algebras}).
\end{cor}

\begin{proof}
For the first half part, by Eq.~(\ref{eq:thm:commutative
2-cocycles and anti-pre-Lie algebras}), we have
$$\mathcal{B}([x,y],z)=\mathcal{B}(x\circ y,z)-\mathcal{B}(y\circ x,z)=\mathcal{B}(y,[x,z])-\mathcal{B}(x,[y,z]),\forall x,y,z\in A.$$
Thus $\mathcal{B}$ is a commutative 2-cocycle on the sub-adjacent
Lie algebra $(\frak g(A),[-,-])$. The second half part follows
from Theorem~\ref{thm:commutative 2-cocycles and anti-pre-Lie
algebras}.
\end{proof}

Recall that two representations $(\rho_{1},V_{1})$ and
$(\rho_{2},V_{2})$ of a Lie algebra $(\frak g,[-,-])$ are called
\textbf{equivalent} if there is a linear isomorphism
$\varphi:V_{1}\rightarrow V_{2}$ such that
$\varphi(\rho_{1}(x)v)=\rho_{2}(x)\varphi(v), \forall x\in \frak
g, v\in V_{1}$. 


\begin{pro}\label{pro:equivalence and invariance}
Let $(A,\circ)$ be an anti-pre-Lie algebra. Then there is a
nondegenerate invariant bilinear form on $(A,\circ)$ if and only
if $(-\mathcal{L}_{\circ}, A)$ and $(\mathrm{ad}^{*},A^{*})$ are
equivalent as representations of the sub-adjacent Lie algebra
$(\frak g(A),[-,-])$.
\end{pro}

\begin{proof}
Suppose  $\varphi:A\rightarrow A^{*}$ is the linear isomorphism
satisfying
$$\varphi (-\mathcal L_\circ(x)y)={\rm
ad}^*(x)\varphi(y),\;\;\forall x,y\in A.$$ Define a nondegenerate
bilinear form $\mathcal{B}$ on $A$ by
\begin{equation}\label{eq:pro:equivalence and invariance}
\mathcal{B}(x,y)=\langle \varphi(x),y\rangle,\;\; \forall x,y\in
A.
\end{equation}
Then we have
$$\mathcal{B}(x\circ y, z)=\langle\varphi(\mathcal{L}_{\circ}(x)y), z\rangle=-\langle\mathrm{ad}^{*}(x)\varphi(y), z\rangle=\langle\varphi(y),[x,z]\rangle=\mathcal{B}(y,[x,z]),\;\;\forall x,y,z\in A.$$
Thus $\mathcal{B}$ is invariant on $(A,\circ)$.

Conversely, suppose that $\mathcal{B}$ is a nondegenerate
invariant bilinear form on $(A,\circ)$. Define a linear map
$\varphi:A\rightarrow A^{*}$ by Eq.~(\ref{eq:pro:equivalence and
invariance}). Then by a similar proof as above, we  show that
$\varphi$ gives the equivalence between $(-\mathcal{L}_{\circ},A)$
and $(\mathrm{ad}^{*},A^{*})$ as representations of $(\frak
g(A),[-,-])$.
\end{proof}

\begin{cor}
Let $(\frak g,[-,-])$ be a Lie algebra. If there is a nondegenerate commutative 2-cocycle
on $(\frak g,[-,-])$, then there is a compatible anti-pre-Lie algebra $(\frak g,\circ)$ given by
Eq.~(\ref{eq:thm:commutative 2-cocycles and anti-pre-Lie
algebras}) and moreover, $(-\mathcal{L}_{\circ}, \frak g)$ and $(\mathrm{ad}^{*},\mathfrak{g}^{*})$ are
equivalent as representations of $(\frak g,[-,-])$. Conversely, if there is a compatible anti-pre-Lie algebra $(\frak g,\circ)$ such that
$(-\mathcal{L}_{\circ}, \mathfrak{g})$ and $(\mathrm{ad}^{*},\mathfrak{g}^{*})$ are
equivalent as representations of $(\frak g,[-,-])$, then there is a nondegenerate bilinear form $\mathcal B$ satisfying
\begin{equation}
\mathcal B([x,y],z)+\mathcal B(y,[z,x])+\mathcal B(x,[y,z])=0,\;\;\forall x,y,z\in \frak g.
\end{equation}
\end{cor}

\begin{proof}
It follows from Proposition~\ref{pro:equivalence and invariance} and Corollary~\ref{cor:corres}, and their proofs.
\end{proof}

Let $(\rho,V)$ be a representation of a Lie algebra $(\frak
g,[-,-])$. Then there is a Lie algebra structure on  the direct
sum $\frak g\oplus V$ of vector spaces (the {\bf semi-direct
product}) defined by
$$
[x+u,y+v]=[x,y]+\rho(x)v-\rho(y)u,\;\;\forall x,y\in \frak g,
u,v\in V.
$$
We denote it by $\frak g\ltimes_{\rho}V$. Furthermore, there is the
following construction of
nondegenerate commutative 2-cocycles from anti-pre-Lie algebras.

\begin{pro}\label{pro:commutative 2-cocycles on semi-direct Lie algebras}
Let $(A,\circ)$ be an anti-pre-Lie algebra and $(\frak
g(A),[-,-])$ be the sub-adjacent Lie algebra. Define a bilinear
form $\mathcal{B}$ on $A\oplus A^{*}$ by
\begin{equation}\label{eq:pro:commutative 2-cocycles on semi-direct Lie algebras1}
\mathcal{B}(x+a^{*},y+b^{*})=\langle x,b^{*}\rangle+\langle
a^{*},y\rangle, \forall x,y\in A, a^{*},b^{*}\in A^{*}.
\end{equation}
Then $\mathcal{B}$ is a nondegenerate commutative 2-cocycle on the
Lie algebra $\frak g(A)\ltimes_{-\mathcal{L}^{*}_{\circ}}A^*$.
Conversely, let $(\frak g,[-,-])$ be a Lie algebra and
$(\rho,\frak g^*)$ be a representation. Suppose that the bilinear
form given by Eq.~(\ref{eq:pro:commutative 2-cocycles on
semi-direct Lie algebras1}) is a commutative 2-cocycle on $\frak
g\ltimes_{\rho}\frak g^{*}$. Then there is a compatible
anti-pre-Lie algebra structure $(\frak g,\circ)$ on $(\frak
g,[-,-])$ such that $\rho=-\mathcal{L}^{*}_{\circ}$.
\end{pro}

\begin{proof}
It is straightforward to show that $\mathcal{B}$ is a
nondegenerate commutative 2-cocycle on the Lie algebra $\frak
g(A)\ltimes_{-\mathcal{L}^{*}_{\circ}}A^*$. \delete{ Let $x,y,z\in
A, a^{*},b^{*},c^{*}\in A^{*}$. Then
we have
\begin{eqnarray*}
&&\mathcal{B}([x+a^{*},y+b^{*}],z+c^{*})
=\langle[x,y],c^{*}\rangle+\langle b^{*},x\circ z\rangle-\langle
a^{*},y\circ z\rangle,\\
&&\mathcal{B}([y+b^{*},z+c^{*}],x+a^{*})=\langle
[y,z],a^{*}\rangle+\langle c^{*},y\circ x\rangle-\langle
b^{*},z\circ x\rangle,\\
&&\mathcal{B}([z+c^{*},x+a^{*}],y+b^{*})=\langle[z,x],b^{*}\rangle+\langle
a^{*},z\circ y\rangle-\langle c^{*},x\circ y\rangle.
\end{eqnarray*}
 Thus
$$\mathcal{B}([x+a^{*},y+b^{*}],z+c^{*})+\mathcal{B}([y+b^{*},z+c^{*}],x+a^{*})+\mathcal{B}([z+c^{*},x+a^{*}],y+b^{*})=0.$$
}
Conversely, 
by Theorem~\ref{thm:commutative 2-cocycles and anti-pre-Lie
algebras}, there is a compatible anti-pre-Lie algebra structure
given by a bilinear operation $\circ$ on $\frak
g\ltimes_{\rho}\frak g^{*}$ defined by
Eq.~(\ref{eq:thm:commutative 2-cocycles and anti-pre-Lie
algebras}). In particular, we have
$$\mathcal B(x\circ y,z)=\mathcal B(y,[x,z])=0,\;\;\forall
x,y,z\in \frak g.$$ So {for all} $x,y\in \frak g$, $x\circ y\in
\frak g$ and thus $(\frak g,\circ)$ is a compatible anti-pre-Lie
algebra on $(\frak g,[-,-])$. Moreover,
\begin{equation*}
\langle y, -\mathcal L^*_\circ(x) a^*\rangle=\langle x\circ y,
a^{*}\rangle=\mathcal B(x\circ y, a^*)=\mathcal
B(y,[x,a^*])=\langle y,\rho(x)a^{*}\rangle, \;\; \forall x,y\in
\frak g, a^{*}\in \frak g^{*}.
\end{equation*}
Hence $\rho=-\mathcal{L}^{*}_{\circ}$.
\end{proof}

At the end of this subsection, we consider a construction of
anti-pre-Lie algebras from symmetric bilinear forms, where in
particular these bilinear forms are invariant.

\begin{pro}\label{ex:anti-pre-Lie algebras from symmetric bilinear forms}
Let $\mathcal{B}$ be a symmetric bilinear form on a vector space
$A$ and $s\in A$ be a fixed vector. Define a bilinear operation
$\circ:A\otimes A\rightarrow A$ by
\begin{equation}\label{eq:ex:anti-pre-Lie algebras from symmetric bilinear forms}
x\circ y=\mathcal{B}(x,y)s-\mathcal{B}(x,s)y,\;\;\forall x,y\in A.
\end{equation}
 Then $(A,\circ)$ is an anti-pre-Lie
algebra. Moreover, $\mathcal{B}$ is invariant on $(A,\circ)$ and
thus $\mathcal{B}$ is a commutative 2-cocycle on the sub-adjacent
Lie algebra $(\frak g(A),[-,-])$.
\end{pro}

\begin{proof}
    It is straightforward.
\end{proof}

\delete{
\begin{proof}
Let $x,y,z\in A$. Then 
we have
\begin{eqnarray*}
x\circ(y\circ z)
&=&\mathcal{B}(y,z)\mathcal{B}(x,s)s-\mathcal{B}(y,z)\mathcal{B}(x,s)s-\mathcal{B}(y,s)\mathcal{B}(x,z)s+\mathcal{B}(y,s)\mathcal{B}(x,s)z\\
\mbox{} &=&-\mathcal{B}(y,s)\mathcal{B}(x,z)s+\mathcal{B}(y,s)\mathcal{B}(x,s)z,\\
\mbox{}[x,y]\circ z&=&(\mathcal{B}(y,s)x-\mathcal{B}(x,s)y)\circ
z=
-\mathcal{B}(x,s)\mathcal{B}(y,z)s+\mathcal{B}(y,s)\mathcal{B}(x,z)s.
\end{eqnarray*}
Therefore we have
\begin{eqnarray*}
&&x\circ(y\circ z)-y\circ(x\circ z)-[y,x]\circ
z\\
&&=\mathcal{B}(x,s)\mathcal{B}(y,z)s-\mathcal{B}(y,s)\mathcal{B}(x,z)s-\mathcal{B}(x,s)\mathcal{B}(y,z)s+\mathcal{B}(y,s)\mathcal{B}(x,z)s=0,\\
&&[x,y]\circ z+[y,z]\circ x+[z,x]\circ y\\
&&=-\mathcal{B}(x,s)\mathcal{B}(y,z)s+\mathcal{B}(y,s)\mathcal{B}(x,z)s-\mathcal{B}(y,s)\mathcal{B}(z,x)s+\mathcal{B}(z,s)\mathcal{B}(y,x)s\\
&&\hspace{0.4cm}
-\mathcal{B}(z,s)\mathcal{B}(x,y)s+\mathcal{B}(x,s)\mathcal{B}(z,y)s=0.
\end{eqnarray*}
Hence $(A,\circ)$ is an anti-pre-Lie algebra. Moreover, we have
$$\mathcal{B}(y,[x,z])=\mathcal{B}(z,s)\mathcal{B}(y,x)-\mathcal{B}(x,s)\mathcal{B}(y,z)=\mathcal{B}(x\circ y,z),\;\;\forall x,y,z\in A.$$
Therefore $\mathcal B$ is invariant on $(A,\circ)$.
\end{proof}
}


\begin{lem}$($\cite{Bai2004}$)$\label{lem}
Let $(A,\circ)$ be a complex vector space with
a bilinear operation $\circ:A\otimes A\rightarrow A$. Suppose that
the following conditions are satisfied.
\begin{enumerate}
\item $A=A_{1}\oplus A_{2}$ as the direct sum of two vector spaces
and $A_{1}$ is a subalgebra. \item For any $x\in A_1$, the action
of $\mathcal{L}_{\circ}(x)$ and $\mathcal{R}_{\circ}(x)$ on
$A_{2}$ is zero or $\mathrm{id}$, where $\mathcal{R}_{\circ}(x)
(y)=y\circ x, \forall y\in A_2$. \item There exists a non-zero
vector $v\in A_{1}$ such that $x\circ y=y\circ x\in\mathbb{C}v,
\forall x,y\in A_{2}$.
\end{enumerate}
Then the classification of the algebraic operation in $A_{2}$
(without changing other products) is given by the classification
of symmetric bilinear forms on an $n$-dimensional complex vector
space, where $n=\mathrm{dim}A_{2}$. That is, there exists a basis
$\{ e_{1},\cdots, e_{n}\}$ of $A_{2}$ such that the classification
is given as follows: $A_{2}$ is trivial or for every
$k=1,\cdots,n$:
$$e_{i}\circ e_{j}=
\begin{cases}
\delta_{ij}v& i,j=1,\cdots,k;\\
0& \text{otherwise.}
\end{cases} $$
\end{lem}

\begin{pro}\label{pro:classification}
Let $(A,\circ)$ be a complex anti-pre-Lie algebra given in
Proposition~ \ref{ex:anti-pre-Lie algebras from symmetric bilinear
forms} with a basis $\{e_{1},\cdots,e_{n}\}$. Then $(A,\circ)$ is
isomorphic to one of the following mutually non-isomorphic cases (only non-zero products
are given):
\begin{enumerate}
\item[$(B1)_{\mbox{\ }}$] $(A,\circ)$ is trivial;

\item[$(B2)_k$] ($k\in\{ 2,\cdots, n\}$) $e_{j}\circ e_{j}=e_{1},
j=2,\cdots, k$;

\item[$(B3)_{\mbox{\ }}$] $e_{1}\circ e_{j}=e_{j}, j=2,\cdots, n$;

\item[$(B4)_k$]  ($k\in\{ 2,\cdots, n\}$) $e_{1}\circ
e_{j}=e_{j},e_{l}\circ e_{l}=e_{1}, j=2,\cdots, n, l=2,\cdots, k$;

\item[$(B5)_{\mbox{\ }}$] $e_{1}\circ e_{2}=-e_{1},e_{2}\circ
e_{j}=e_{j},
 j=2,\cdots, n$;


\item[$(B6)_k$] ($k\in\{ 3,\cdots, n\}$) $e_{1}\circ
e_{2}=-e_{1},e_{2}\circ e_{2}=e_{2}, e_{2}\circ e_{j}=e_{j},
e_{l}\circ e_{l}=e_{1}, j=3,\cdots, n, l=3,\cdots, k$.
\end{enumerate}
\end{pro}


\begin{proof}
Without losing generality, we assume that $s=e_{1}$. If
$\mathrm{dim}A=1$, then we get $(B1)$. Next we assume that
$\mathrm{dim}A\geq 2$. Set
$$b_{ij}=b_{ji}=\mathcal{B}(e_{i},e_{j}),\;\; \forall i,j=1,\cdots,
n.$$ Then by Eq.~(\ref{eq:ex:anti-pre-Lie algebras from symmetric
bilinear forms}), we have
\begin{equation}\label{eq:classification of bilinear form}
e_{i}\circ e_{j}=b_{ij}e_{1}-b_{1i}e_{j}, \ \ i,j=1,\cdots, n.
\end{equation}
\begin{enumerate}
\item[(1)] Suppose that
$b_{1i}=\mathcal{B}(e_{1},e_{i})=0,\;\;\forall i=1,\cdots, n$.
Then by Eq.~(\ref{eq:classification of bilinear form}), the
anti-pre-Lie algebra $(A,\circ)$ is given by the following
non-zero products:
$$e_{i}\circ e_{j}=b_{ij}e_{1}, \ \ i,j=2,\cdots, n.$$
Let $A_{1}=\mathbb{C}e_{1}$ be a subalgebra and $A_{2}$ be a
subspace spanned by $e_{2},\cdots, e_{n}$. Then by Lemma
\ref{lem}, the classification is given by $(B1)$ and $(B2)$.

\item[(2)] Suppose that $b_{11}=\mathcal{B}(e_{1},e_{1})\neq 0.$
Then there are vectors $e_{2},\cdots, e_{n}$ such that
$\{e_{1},\cdots, e_{n}\}$ is a basis and
$b_{1j}=\mathcal{B}(e_{1},e_{j})=0,\;\forall j=2,\cdots, n$. Then
by Eq.~(\ref{eq:classification of bilinear form}), the
anti-pre-Lie algebra $(A,\circ)$ is given by the following
non-zero products:
$$e_{1}\circ e_{j}=-b_{11}e_{j}, e_{j}\circ e_{k}=b_{jk}e_{1}, \ \  j,k=2,\cdots, n.$$
Let $e_{1}\rightarrow -\dfrac{1}{b_{11}}e_{1}, e_{j}\rightarrow
\dfrac{\mathrm{1}}{\sqrt{-b_{11}}}e_{j}, j=2,\cdots, n$. Then we
can assume that $b_{11}=-1$. Let $A_{1}=\mathbb{C}e_{1}$ be a
subalgebra and $A_{2}$ be a subspace spanned by $e_{2},\cdots,
e_{n}$. Then by Lemma \ref{lem}, the classification is given by
$(B3)$ and $(B4)$.

\item[(3)] Suppose that $b_{11}=\mathcal{B}(e_{1}, e_{1})=0$. Then
there exists a $u\in A$ such that $\mathcal{B}(u,e_{1})\neq 0$.
Thus we let $u=e_{2}$ and if $\mathrm{dim}A>2$, then
 there are vectors $e_{3},\cdots, e_{n}$
 such that $\{ e_{1}, \cdots, e_{n}\}$ is a
basis and $b_{1j}=\mathcal{B}(e_{1},e_{j})=0,\;\forall j=3,\cdots,
n$. Thus by Eq.~(\ref{eq:classification of bilinear form}), the
anti-pre-Lie algebra $(A,\circ)$ is given by the following
non-zero products:
$$e_{1}\circ e_{2}=b_{12}e_{1}, e_{2}\circ e_{2}=b_{22}e_{1}-b_{12}e_{2}, e_{2}\circ e_{j}=b_{2j}e_{1}-b_{12}e_{j}, $$
$$e_{j}\circ e_{2}=b_{2j}e_{1}, e_{j}\circ e_{l}=b_{jl}e_{l},\ \  j,l=3,\cdots, n.$$
Let
$$e_{1}\rightarrow e_{1}, e_{2}\rightarrow \dfrac{b_{22}}{2b^{2}_{12}}e_{1}-\dfrac{1}{b_{12}}e_{2}, e_{j}\rightarrow -\dfrac{b_{2j}}{b_{12}}e_{1}+e_{j}, \ \ j=3,\cdots, n. $$
Under the new basis, the non-zero products are given by:
$$e_{1}\circ e_{2}=-e_{1}, e_{2}\circ e_{2}=e_{2}, e_{2}\circ e_{j}=e_{j}, e_{j}\circ e_{l}=b_{jl}e_{1}, \ \ j,l=3,\cdots, n.$$
Let $A_{1}$ be a subalgebra spanned by $e_{1},e_{2}$, and $A_{2}$
be a subspace spanned by $e_{3},\cdots, e_{n}$. Then by Lemma
\ref{lem}, the classification is given by $(B5)$ and $(B6)$.
\end{enumerate}
Hence the conclusion holds.
\end{proof}

\begin{rmk}
Among the anti-pre-Lie algebras given in
Proposition~\ref{pro:classification}, we have the following
results.
\begin{enumerate}
\item Associative algebras: $(B1)$ and $(B2)_k$ ($k\in \{2,\cdots,
n\}$). All of them are commutative.
\item Anti-associative
algebras: $(B1)$ and $(B2)_k$ ($k\in \{2,\cdots, n\}$).
\item
Pre-Lie algebras: $(B1)$, $(B2)_k$ ($k\in \{2,\cdots, n\}$) and
$(B3)$.
\item Suppose the bilinear form in Eq.~(\ref{eq:ex:anti-pre-Lie algebras from symmetric bilinear forms}) is nondegenerate. If $\mathrm{dim}A=1$, then $(A,\circ)$ is isomorphic to $(B1)$; if $\mathrm{dim}A=2$, then $(A,\circ)$ is isomorphic to $(B4)_{2}$ or $(B5)$; if $\mathrm{dim}A=n\geq 3$, then $(A,\circ)$ is isomorphic to $(B4)_{n}$ or $(B6)_{n}$.
\item If $\mathrm{dim}A=2$, then $(B3)$ is isomorphic to
$(A2)_0$, $(B4)_2$ is isomorphic to $(A1)$, and $(B5)$ is
isomorphic to $(A4)_0$, where we use the notations in
Proposition~\ref{thm:classification}.
\end{enumerate}
\end{rmk}


\subsection{Comparison between anti-pre-Lie algebras and pre-Lie algebras} We summarize the study in the previous subsections to
exhibit a clear analogy between anti-pre-Lie algebras and pre-Lie
algebras.

The study in the previous subsections allows us to compare
anti-pre-Lie algebras and pre-Lie algebras in terms of the
following properties (for the part involving pre-Lie algebras, see
\cite{Bai2004,Bai2006,Bai2007,Chu,Ku1,SS}):
\begin{enumerate}
\item Representations of the sub-adjacent Lie algebras, giving the
characterization of the algebraic structures;

\item Operators on Lie algebras, giving the characterization and
construction;

\item Bilinear forms on the sub-adjacent Lie algebras, giving the
close relationships between the nondegenerate ones and the
compatible algebraic structures;

\item Construction from linear functions;

\item Construction from symmetric bilinear forms;

\item Compatibility on simple Lie algebras.

\end{enumerate}
 We list them in Table 1. From this table, we observe
that there is a clear analogy between anti-pre-Lie algebras and
pre-Lie algebras.

\begin{table}[h]
\small
\begin{longtable}{|c|c|c|}
\multicolumn{3}{c}{\textbf{Table 1. Comparison between anti-pre-Lie algebras and pre-Lie algebras}}\\
\hline
Algebras & Anti-pre-Lie algebras $(A,\circ)$ & Pre-Lie algebras $(A,\star)$\\
\hline
\tabincell{c}{Representations of the \\ sub-adjacent Lie algebras} & $(-\mathcal{L}_{\circ},A)$ & $(\mathcal{L}_{\star},A)$\\
\hline
\tabincell{c}{Operators on Lie algebras} & (strong) anti-$\mathcal{O}$-operators & $\mathcal{O}$-operators\\
\hline
\multirow{4}{*}{\tabincell{c}{Bilinear forms on the \\ sub-adjacent Lie algebras}} & \multicolumn{2}{c|}{$\mathcal{B}([x,y],z)+\mathcal{B}([y,z],x)+\mathcal{B}([z,x],y)=0$}\\
\cline{2-3}
 & symmetric & skew-symmetric\\
 \cline{2-3}
 & $\mathcal{B}(x\circ y,z)=\mathcal{B}(y,[x,z])$ & $\mathcal{B}(x\star y,z)=-\mathcal{B}(y,[x,z])$\\
 \cline{2-3}
 & $\frak g(A)\ltimes_{-\mathcal L_\circ^*}A^*$ & $\frak g(A)\ltimes_{\mathcal L_\star^*}A^*$\\
\hline
\multirow{2}{*}{\tabincell{c}{Construction from \\ linear functions}}
& $x\circ y=g(x)y$ & $x\star y=g(x)y$\\
 \cline{2-3}
& $x\circ y= f(y)x+2f(x)y$ & $x\star y=f(y)x$\\
\hline
\tabincell{c}{Construction from \\ symmetric bilinear forms} & $x\circ y=\mathcal{B}(x,y)s-\mathcal{B}(x,s)y$ & $x\star y=\mathcal{B}(x,y)s+\mathcal{B}(x,s)y$\\
\hline
\tabincell{c}{Compatibility on \\ simple Lie algebras} & Yes & No\\
\hline
\end{longtable}
\end{table}

\section{Novikov algebras and admissible Novikov algebras}

We introduce the notion of admissible Novikov algebras as a
subclass of anti-pre-Lie algebras, corresponding to Novikov
algebras in terms of $q$-algebras (\cite{Dzh09}). Such a
correspondence gives the construction of anti-pre-Lie algebras
from commutative associative algebras with derivations or
admissible pairs and leads to the introduction of the notions of anti-pre-Lie
Poisson algebras and admissible Novikov-Poisson algebras.

\subsection{Correspondence between Novikov algebras and admissible Novikov algebras}

We introduce the notion of admissible Novikov algebras as a
subclass of anti-pre-Lie algebras, whose name comes from a
correspondence between them and Novikov algebras in terms of
$q$-algebras. The interpretation of admissible Novikov algebras in
terms of the corresponding anti-$\mathcal O$-operators and some
examples are given.

\begin{defi}(\cite{Bal, Gel})\label{defi:Novikov algebras}
A \textbf{Novikov algebra} is a pre-Lie algebra $(A,\star)$ such that
\begin{equation}\label{eq:defi:Novikov algebras1}
(x\star y)\star z=(x\star z)\star y, \forall x,y,z\in A.
\end{equation}
\end{defi}

\begin{defi}\label{defi:admissible Novikov algebras} Let $A$ be a
vector space with a bilinear operation $\circ:A\otimes A
\rightarrow A$. $(A,\circ)$ is called an \textbf{admissible
Novikov algebra} if Eq.~(\ref{eq:defi:anti-pre-Lie algebras1}) and
the following equation hold:
\begin{equation}\label{eq:defi:admissible Novikov algebras1}
2x\circ[y,z]=(x\circ y)\circ z-(x\circ z)\circ y, \;\; \forall
x,y,z\in A.
\end{equation}
\end{defi}

\begin{pro}\label{pro:admissible is anti}
An admissible Novikov algebra is an anti-pre-Lie algebra.
\end{pro}
\begin{proof}
Let $(A,\circ)$ be an admissible Novikov algebra, and $x,y,z\in
A$. By Eq.~(\ref{eq:defi:anti-pre-Lie algebras1}), Eq.~(\ref{eq:equal}) holds, that is,
we have
$$x\circ[y,z]+y\circ[z,x]+z\circ[x,y]=[y,x]\circ z+[z,y]\circ x+[x,z]\circ y.$$
On the other hand, by Eq.~(\ref{eq:defi:admissible Novikov
algebras1}), we have
\begin{eqnarray*}
&&2x\circ[y,z]+2y\circ[z,x]+2z\circ[x,y]\\
&&=(x\circ y)\circ z-(x\circ z)\circ y+(y\circ z)\circ x-(y\circ x)\circ z+(z\circ x)\circ y-(z\circ y)\circ x\\
&&=[x,y]\circ z+[y,z]\circ x+[z,x]\circ y.
\end{eqnarray*}
Thus
$$x\circ[y,z]+y\circ[z,x]+z\circ[x,y]=[x,y]\circ z+[y,z]\circ x+[z,x]\circ y=0. $$
Hence $(A,\circ)$ is an anti-pre-Lie algebra.
\end{proof}

We introduce the notion of ``admissible Novikov algebras" since
they ``correspond" to Novikov algebras in the following sense:
Novikov algebras and admissible Novikov algebras can be realized
as $q$-algebras  each other. Note that the $q$-algebras are used
in the classification of algebras with skew-symmetric identities
of degree 3 (\cite{Dzh09}).

\begin{defi} (\cite{Dzh09})\label{defi:q-algebra}
Let $A$ be a vector space with a bilinear operation
$\star:A\otimes A \rightarrow A$. Define a bilinear operation
$\circ:A\otimes A\rightarrow A$ by
\begin{equation}\label{eq:defi:q-algebra}
x\circ y=x\star y+qy\star x, \forall x,y\in A,
\end{equation}
for some $q\in\mathbb{F}$. Then $(A,\circ)$ is called the
\textbf{$q$-algebra} of $(A,\star)$.
\end{defi}

\begin{pro}\label{pro:anti-pre-Lie algebras from Novikov algebras}
Let $(A,\star)$ be a pre-Lie algebra, and $(A,\circ)$ be the
2-algebra of $(A,\star)$, that is,
\begin{equation}\label{eq:pro:anti-pre-Lie algebras from Novikov algebras}
x\circ y=x\star y+2y\star x,\forall x,y\in A.
\end{equation}
Then $(A,\circ)$ is an anti-pre-Lie algebra if and only if
$(A,\star)$ is further a Novikov algebra. Moreover, in this case,
$(A,\circ)$ is an admissible Novikov algebra.
\end{pro}
\begin{proof}
Let the sub-adjacent Lie algebra of $(A,\star)$ be $(\frak
g(A),\lbrace-,-\rbrace)$. Then by Eq.~(\ref{eq:pro:anti-pre-Lie
algebras from Novikov algebras}), we have
\begin{equation*}\label{eq:pro:anti-pre-Lie algebras from Novikov algebras2}
[x,y]=x\circ y-y\circ x=y\star x-x\star y=\lbrace y,x\rbrace,
\forall x,y\in A.
\end{equation*}
Thus $(A,\circ)$ is a Lie-admissible algebra. 
Let $x,y,z\in A$. Then we have
\begin{eqnarray*}
&&x\circ(y\circ z)-y\circ(x\circ z)+[x,y]\circ z\\
&&\overset{(\ref{eq:pro:anti-pre-Lie algebras from Novikov algebras})}{=}x\star(y\star z)+2(y\star z)\star x+2x\star(z\star y)+4(z\star y)\star x-y\star(x\star z)-2(x\star z)\star y\\
&&\hspace{0.4cm}-2y\star(z\star x)-4(z\star x)\star y+(y\star x)\star z+2z\star(y\star x)-(x\star y)\star z-2z\star(x\star y)\\
&&\overset{(\ref{eq:defi:pre-Lie algebras})}{=}2(y\star z)\star x-2y\star(z\star x)+2z\star(y\star x)-2(x\star z)\star y+2x\star(z\star y)-2z\star(x\star y)\\
&&\hspace{0.4cm}+4(z\star y)\star x-4(z\star x)\star y\\
&&\overset{(\ref{eq:defi:pre-Lie algebras})}{=}6(z\star y)\star
x-6(z\star x)\star y.
\end{eqnarray*}
Then $(A,\circ)$ is an anti-pre-Lie algebra if and only if
$(A,\star)$ is a Novikov algebra. Moreover, if  $(A,\star)$ is a
Novikov algebra, then we have
\begin{eqnarray*}
&&2x\circ[y,z]-(x\circ y)\circ z+(x\circ z)\circ y\\
&&=2x\circ(z\star y)-2x\circ(y\star z)-(x\circ y)\circ z+(x\circ z)\circ y\\
&&\overset{(\ref{eq:pro:anti-pre-Lie algebras from Novikov
algebras})}{=}2x\star(z\star y)+4(z\star y)\star x-2x\star(y\star
z)-4(y\star z)\star x
-(x\star y)\star z-2z\star(x\star y)\\
&&\hspace{0.4cm}-2(y\star x)\star z-4z\star(y\star x)+(x\star z)\star y+2y\star(x\star z)+2(z\star x)\star y+4y\star(z\star x)\\
&&\overset{(\ref{eq:defi:pre-Lie algebras}),(\ref{eq:defi:Novikov
algebras1})}{=}2(x\star z)\star y-2(x\star y)\star z=0.
\end{eqnarray*}
Thus $(A,\circ)$ is an admissible Novikov algebra.
\end{proof}


By Eq.~(\ref{eq:pro:anti-pre-Lie algebras from Novikov algebras}),
there is an equivalent expression:
\begin{equation}\label{eq:pro:admissible Novikov algebras and Novikov algebras1}
x\star y=-\dfrac{1}{3}x\circ y+\dfrac{2}{3}y\circ x, \forall
x,y\in A.
\end{equation}
In terms of $q$-algebras, we adjust the above equation to be
\begin{equation}\label{eq:pro:admissible Novikov algebras and Novikov algebras2}
x\star y=x\circ y-2y\circ x, \forall x,y\in A.
\end{equation}

\begin{pro}\label{pro:admissible Novikov algebras and Novikov algebras}
Let $(A,\circ)$ be an anti-pre-Lie algebra, and $(A,\star)$ be the
(-2)-algebra of $(A,\circ)$ given by Eq.~(\ref{eq:pro:admissible
Novikov algebras and Novikov algebras2}).
Then $(A,\star)$ is a pre-Lie algebra if and only if $(A,\circ)$
is further an admissible Novikov algebra. Moreover, in this case,
$(A,\star)$ is a Novikov algebra.
\end{pro}
\begin{proof}
Let $x,y,z\in A$. Then we have
\begin{eqnarray*}
&&(x\star y)\star z-x\star(y\star z)-(y\star x)\star z+y\star(x\star z)\\
&&\overset{(\ref{eq:pro:admissible Novikov algebras and Novikov
algebras2})}{=}(x\circ y)\circ z-2z\circ(x\circ y)-2(y\circ
x)\circ z+4z\circ(y\circ x)-x\circ(y\circ z)+2(y\circ z)\circ x\\
&&\hspace{0.4cm}+2x\circ(z\circ y)-4(z\circ y)\circ x-(y\circ x)\circ z+2z\circ(y\circ x)+2(x\circ y)\circ z-4z\circ(x\circ y)\\
&&\hspace{0.4cm}+y\circ(x\circ z)-2(x\circ z)\circ y-2y\circ(z\circ x)+4(z\circ x)\circ y\\
&&\overset{(\ref{eq:defi:anti-pre-Lie algebras1})}{=}4[x,y]\circ z-6z\circ[x,y]+2[y,z]\circ x+2[z,x]\circ y-2(z\circ y)\circ x+2(z\circ x)\circ y\\
&&\hspace{0.4cm}+2x\circ(z\circ y)-2y\circ(z\circ x)\\
&&\overset{(\ref{eq:defi:anti-pre-Lie algebras5})}{=}4[x,y]\circ z-6z\circ[x,y]+2[y,z]\circ x+2[z,x]\circ y-2(z\circ y)\circ x+2(z\circ x)\circ y\\
&&\hspace{0.4cm}+2[y,x]\circ z-2z\circ[y,x]\\
&&\overset{(\ref{eq:defi:anti-pre-Lie
algebras2})}{=}-4z\circ[x,y]+2(z\circ x)\circ y-2(z\circ y)\circ
x.
\end{eqnarray*}
Thus $(A,\star)$ is a pre-Lie algebra if and only if $(A,\circ)$
is an admissible Novikov algebra. Moreover, if $(A,\circ)$ is an
admissible Novikov algebra, then we have
\begin{eqnarray*}
&&(x\star y)\star z-(x\star z)\star y\\
&&\overset{(\ref{eq:pro:admissible Novikov algebras and Novikov
algebras2})}{=}(x\circ y)\circ z-2z\circ(x\circ y)-2(y\circ
x)\circ z+4z\circ(y\circ x)-(x\circ z)\circ y+2y\circ(x\circ z)\\
&&\hspace{0.4cm}+2(z\circ x)\circ y-4y\circ(z\circ x)\\
&&\overset{(\ref{eq:defi:anti-pre-Lie algebras3})}{=}[x,y]\circ z-(y\circ x)\circ z-2z\circ[x,y]+2z\circ(y\circ x)-[x,z]\circ y+(z\circ x)\circ y\\
&&\hspace{0.4cm}+2y\circ[x,z]-2y\circ(z\circ x)\\
&&\overset{(\ref{eq:defi:anti-pre-Lie algebras1})}{=}[x,y]\circ z+[z,x]\circ y+2[y,z]\circ x+2z\circ[y,x]+2y\circ[x,z]-(y\circ x)\circ z+(z\circ x)\circ y\\
&&\overset{(\ref{eq:defi:anti-pre-Lie algebras2}),(\ref{eq:defi:anti-pre-Lie algebras5})}{=}[x,z]\circ y+[y,x]\circ z+2x\circ[y,z]-(y\circ x)\circ z+(z\circ x)\circ y\\
&&\overset{(\ref{eq:defi:anti-pre-Lie algebras3})}{=}2x\circ
[y,z]+(x\circ z)\circ y-(x\circ y)\circ
z\overset{(\ref{eq:defi:admissible Novikov algebras1})}{=}0.
\end{eqnarray*}
Thus $(A,\star)$ is a Novikov algebra.
\end{proof}

\begin{rmk}
If we only consider the categories of Novikov algebras and
admissible Novikov algebras, then the fact that $(A,\star)$ is a
Novikov algebra if and only if its 2-algebra $(A,\circ)$ is an
admissible Novikov
 algebra implies the ``converse" side that $(A,\circ)$ is an admissible
Novikov algebra if and only if its (-2)-algebra $(A,\star)$ is a
Novikov algebra. However, since the above correspondence is put
into a bigger framework involving pre-Lie algebras and
anti-pre-Lie algebras, we would like to point out that
Proposition~\ref{pro:admissible Novikov algebras and Novikov
algebras} cannot be obtained from
Proposition~\ref{pro:anti-pre-Lie algebras from Novikov algebras}
directly.
\end{rmk}




Propositions \ref{pro:anti-pre-Lie algebras from Novikov algebras}
and \ref{pro:admissible Novikov algebras and Novikov algebras}
give the following relationship, as already illustrated in
Introduction.


$$\mbox{\{pre-Lie\} $\hookleftarrow$ \{Novikov\}
$\overset{2-algebra}{\underset{(-2)-algebra}{\rightleftharpoons}}$
\{admissible Novikov\} $\hookrightarrow$ \{anti-pre-Lie\}.}$$

\begin{ex}
Any commutative associative algebra is both a Novikov algebra and
an admissible Novikov algebra.
\end{ex}

\begin{ex}\label{ex:Novikov algebras from linear functions}
Let $A$ be a vector space of dimension $n\geq 2$. Let $f,g:
A\rightarrow \mathbb{C}$ be two linear functions. Then it is
straightforward to show Eq.~(\ref{eq:anti-pre-Lie algebras from
linear functions}) defines a Novikov algebra $(A,\star)$ if and
only if $g=0$, that is, $x\star y=f(y)x$ for all $x,y\in A$. On
the other hand, Eq.~(\ref{eq:anti-pre-Lie algebras from linear
functions}) defines an admissible Novikov algebra $(A,\circ)$ if
and only if $g=2f$, that is, $x\circ y=f(y)x+2f(x)y$ for all
$x,y\in A$, whose classification is given as Cases (\ref{it:f2})
and (\ref{it:f3}) in Corollary~\ref{cor:class}. The latter can be
obtained directly or from the correspondence with Novikov
algebras. Explicitly, let $x,y\in A$. If $(A,\circ)$ is an
admissible Novikov algebra, then the $(-2)$-algebra $(A,\star)$ of
$(A,\circ)$ is a Novikov algebra, where
$$x\star y=x\circ y-2y\circ x=(f-2g)(y)x+(g-2f)(x)y.$$
Then we have $g-2f=0$. Conversely, if $g=2f$, then $(A,\circ)$ is
the $2$-algebra of the Novikov algebra $(A,\star)$ given by
$x\star y=f(y)x$. Thus $(A,\circ)$ is an admissible Novikov
algebra.
\end{ex}

\begin{ex}
With the notations given in Proposition~\ref{thm:classification},
any complex 2-dimensional non-commutative admissible Novikov
algebra is isomorphic to one of the following cases:
\begin{enumerate}
\item[(AN1):] $(A2)_{-2}$ and $(A4)_\lambda$ with $\lambda\ne -1$,
that is,
$$e_{1}\circ e_{1}=0,e_{1}\circ e_{2}=(\lambda+1)e_{1},
e_{2}\circ e_{1}=\lambda e_{1}, e_{2}\circ e_{2}=(\lambda-1)e_{2},
\lambda\in\mathbb{C};$$ \item[(AN2):] $(A5)$, that is,
$$e_{1}\circ e_{1}=0, e_{1}\circ e_{2}=-e_{1}, e_{2}\circ
e_{1}=-2e_{1}, e_{2}\circ e_{2}=e_{1}-3e_{2}.$$
\end{enumerate}
The $(-2)$-algebras of the above admissible Novikov algebras give
the classification of 2-dimensional non-commutative complex
Novikov algebras:
\begin{enumerate}
\item[(N1)] $e_{1}\star e_{1}=0, e_{1}\star
e_{2}=(1-\lambda)e_{1}, e_{2}\star e_{1}=-(\lambda+2)e_{1},
e_{2}\star e_{2}=(1-\lambda)e_{2},\lambda\in\mathbb{C};$

\item[(N2)] $e_{1}\star  e_{1}=0, e_{1}\star e_{2}=3e_{1},
e_{2}\star e_{1}=0, e_{2}\star e_{2}=-e_{1}+3e_{2}$.
\end{enumerate}
On the other hand, the classification of complex 2-dimensional
non-commutative Novikov algebras is presented in another way in
\cite{Bai2001}.
\begin{enumerate}
\item[(N'1)] $e_{1}\star e_{1}=0, e_{1}\star e_{2}=0, e_{2}\star
e_{1}=-e_{1}, e_{2}\star e_{2}=0;$

\item[(N'2)] $e_{1}\star e_{1}=0, e_{1}\star e_{2}=e_{1},
e_{2}\star e_{1}=0, e_{2}\star e_{2}=e_{2};$

\item[(N'3)] $e_{1}\star e_{1}=0, e_{1}\star e_{2}=e_{1},
e_{2}\star e_{1}=le_{1}, e_{2}\star e_{2}=e_{2},
l\in\mathbb{C}\setminus\lbrace 0,1\rbrace;$

\item[(N'4)] $e_{1}\star e_{1}=0, e_{1}\star e_{2}=e_{1},
e_{2}\star e_{1}=0, e_{2}\star e_{2}=e_{1}+e_{2}.$
\end{enumerate}
The correspondence between the two presentations of Novikov
algebras is given as follows.

$(N1)$ with $\lambda=1$ $\longleftrightarrow$ $(N'1)$;\quad\quad
$(N1)$ with $\lambda=-2$ $\longleftrightarrow$ $(N'2)$;

$(N1)$ with $\lambda\ne 1,-2$ $\longleftrightarrow$ $(N'3)$ with
$l=1-\dfrac{3}{1-\lambda}$;\quad \quad $(N2)$
$\longleftrightarrow$ $(N'4)$.
\end{ex}


The correspondence between Novikov algebras and admissible Novikov
algebras induces the following correspondence on symmetric
bilinear forms.

\begin{pro}\label{pro:bilinear forms on Novikov algebras and admissible Novikov algebras}
Let $\mathcal B$ be a symmetric bilinear form on a Novikov algebra
$(A,\star)$. Assume that $(A,\circ)$ is the corresponding
admissible Novikov algebra, that is, the 2-algebra of $(A,\star)$.
Then $\mathcal B$ is invariant on $(A,\circ)$ if and only if the
following condition is satisfied:
\begin{equation}\label{eq:pro:bilinear forms on Novikov algebras and admissible Novikov algebras1}
\mathcal{B}(x\star y,z)=-\mathcal{B}(y, x\star z+z\star
x),\;\;\forall x,y,z\in A.
\end{equation}
\end{pro}

\begin{proof}
Let $x,y,z\in A$. Suppose that $\mathcal B$ is invariant on
$(A,\circ)$. Then we have
\begin{eqnarray*}
3\mathcal{B}(x\star y,z) &\overset{(\ref{eq:pro:admissible Novikov
algebras and Novikov algebras1})}{=}&-\mathcal{B}(x\circ y,
z)+2\mathcal{B}(y\circ x,z)
\overset{(\ref{eq:thm:commutative 2-cocycles and anti-pre-Lie algebras})}{=}-\mathcal{B}(y,[x,z])+2\mathcal{B}(x,[y,z])\\
&\overset{(\ref{eq:2-cocycle})}{=}&\mathcal{B}(x,[y,z])+\mathcal{B}(z,[y,x]).
\end{eqnarray*}
Due to the symmetry of $x$ and $z$ in the above equation, we have
\begin{equation}\label{eq:bilinear}
\mathcal B(x\star y,z)=\mathcal B(z\star y,x),\;\;\forall x,y,z\in
A.
\end{equation}
Moreover, we have
\begin{eqnarray*}
3\mathcal{B}(x\star y,z)
&=&\mathcal{B}(x,z\star y)-\mathcal{B}(x,y\star z)+\mathcal{B}(z,x\star y)-\mathcal{B}(z,y\star x)\\
&\overset{(\ref{eq:bilinear})}{=}&2\mathcal{B}(z,x\star
y)-\mathcal{B}(y,x\star z)-\mathcal{B}(y,z\star x).
\end{eqnarray*}
Hence Eq.~(\ref{eq:pro:bilinear forms on Novikov algebras and
admissible Novikov algebras1}) holds.

Conversely, if $\mathcal B$ satisfies Eq.~(\ref{eq:pro:bilinear
forms on Novikov algebras and admissible Novikov algebras1}) on
the Novikov algebra $(A,\star)$, then Eq.~(\ref{eq:bilinear})
holds and
\begin{eqnarray*}
\mathcal{B}(x\circ y,z)&\overset{(\ref{eq:pro:anti-pre-Lie
algebras from Novikov algebras})}=&\mathcal{B}(x\star
y,z)+2\mathcal{B}(y\star x,z)=\mathcal B(y\star
x,z)+\mathcal{B}(z,x\star y+y\star x)\\
&\overset{(\ref{eq:pro:bilinear forms on Novikov algebras and
admissible Novikov
algebras1}),(\ref{eq:bilinear})}=&\mathcal{B}(z\star
x,y)-\mathcal{B}(x\star z,y) =\mathcal{B}([x,z],y).
\end{eqnarray*}
Thus $\mathcal B$ is invariant on $(A,\circ)$.
\end{proof}

The correspondence between Novikov algebras and
infinite-dimensional Lie algebras (\cite{Bal}) gives a
correspondence between admissible Novikov algebras and
infinite-dimensional Lie algebras.

\begin{pro} \label{lem:BN} {\rm (\cite{Bal})} Let $A$ be a vector space with a
bilinear operation $\star:A\otimes A\rightarrow A$. Set $\hat
A=A\otimes {\mathbb F}[t,t^{-1}]$. Define a bilinear operation
$[-,-]:\hat A\otimes \hat A\rightarrow \hat A$ by
\begin{equation}
[x\otimes t^{m+1}, y\otimes t^{n+1}]=((m+1)x\star y-(n+1)y\star
x)\otimes t^{m+n+1},\;\;\forall x,y\in A, m,n\in {\mathbb Z}.
\end{equation}
Then $(\hat A,[-,-])$ is a Lie algebra if and only if $(A,\star)$
is a Novikov algebra.
\end{pro}

Combining Propositions~\ref{pro:admissible Novikov algebras and
Novikov algebras} and \ref{lem:BN} together, we have the following
conclusion.

\begin{cor}\label{cor:BN-A}
Let $A$ be a vector space with a bilinear operation
$\circ:A\otimes A\rightarrow A$. Set $\hat A=A\otimes {\mathbb
F}[t,t^{-1}]$. Define a bilinear operation $[-,-]:\hat A\otimes \hat A\rightarrow \hat A$ by
\begin{equation}
[x\otimes t^{m+1}, y\otimes t^{n+1}]=((m+2n+3)x\circ
y-(2m+n+3)y\circ x)\otimes t^{m+n+1},\;\;\forall x,y\in A, m,n\in
{\mathbb Z}.
\end{equation}
Then $(\hat A,[-,-])$ is a Lie algebra if and only if $(A,\star)$
is an admissible Novikov algebra.
\end{cor}

Next we interpret admissible Novikov algebras in terms of
anti-$\mathcal O$-operators.

\begin{defi}\label{defi:admissible anti O-operators}
Let $T:V\rightarrow \frak g$ be an anti-$\mathcal{O}$-operator of
a Lie algebra $(\frak g,[-,-])$ associated to a representation
$(\rho,V)$. Then $T$ is called \textbf{admissible} if {\small
\begin{equation}\label{eq:defi:admissible anti O-operators}
2\rho(T(u))\rho(T(v))w-2\rho(T(u))\rho(T(w))v
=\rho(T(\rho(T(u))v))w-\rho(T(\rho(T(u))w))v,\;\;\forall u,v,w\in
V.
\end{equation}}
In particular, an anti-Rota-Baxter operator $R$ of $(\frak
g,[-,-])$ is called \textbf{admissible} if
\begin{equation}\label{eq:defi:admissible anti-Rota-Baxter perators}
2[R(x),[R(y),z]]-2[R(x),[R(z),y]]=[R([R(x),y]),z]-[R([R(x),z]),y], \forall x,y,z\in A.
\end{equation}
\end{defi}

\begin{pro}\label{pro:admissible operator}
Let $T:V\rightarrow \frak g$ be an admissible
anti-$\mathcal{O}$-operator of a Lie algebra $(\frak g,[-,-])$
associated to a representation $(\rho,V)$.  Then $T$ is strong.
Moreover, $(V,\circ)$ is an admissible Novikov algebra, where the
operation $\circ$ is given by Eq.~(\ref{eq:pro:O-operator}).
\end{pro}
\begin{proof}
Let $u,v,w\in V$. By Eqs.~(\ref{eq:defi:anti O-operators1}) and
(\ref{eq:defi:admissible anti O-operators}), we have
\begin{eqnarray*}
&&2\rho([T(u),T(v)])w+2\rho([T(v),T(w)])u+2\rho([T(w),T(u)])v\\
&&=\rho(T(\rho(T(u))v))w-\rho(T(\rho(T(u))w))v+\rho(T(\rho(T(v))w))u-\rho(T(\rho(T(v))u))w\\
&&\hspace{0.4cm}+\rho(T(\rho(T(w))u))v-\rho(T(\rho(T(w))v))u\\
&&=\rho([T(v),T(u)])w+\rho([T(w),T(v)])u+\rho([T(u),T(w)])v.
\end{eqnarray*}
Hence Eq.~(\ref{eq:defi:anti O-operators2}) holds and thus $T$ is
strong. Moreover, we have
\begin{eqnarray*}
2u\circ[v,w]&=&2\rho(T(u))\rho(T(v))w-2\rho(T(u))\rho(T(w))v=\rho(T(\rho(T(u))v))w-\rho(T(\rho(T(u))w))v\\
&=&(u\circ v)\circ w-(u\circ w)\circ v.
\end{eqnarray*}
Then by Proposition \ref{pro:O-operator}, $(V,\circ)$ is an
admissible Novikov algebra.
\end{proof}

\begin{cor}
Let $(\mathfrak{g},[-,-])$ be a Lie algebra, and
$R:\mathfrak{g}\rightarrow \mathfrak{g}$ be an admissible
anti-Rota-Baxter operator. Then $(A,\circ)$ given by
Eq.~(\ref{eq:cons}) defines an admissible Novikov algebra.
Conversely, if $R:\mathfrak{g}\rightarrow \mathfrak{g}$ is a
linear transformation on a Lie algebra $(\mathfrak{g},[-,-])$ such
that Eq.~(\ref{eq:cons}) defines an admissible Novikov algebra,
then $R$ satisfies Eqs.~(\ref{eq:cons2}) and
(\ref{eq:defi:admissible anti-Rota-Baxter perators}).
\end{cor}
\begin{proof}
The first half part follows from Proposition \ref{pro:admissible
operator} by letting $\rho=\mathrm{ad}$. The second half part
follows from Definition~\ref{defi:admissible Novikov algebras}
directly.
\end{proof}

\begin{ex}
Among the anti-Rota-Baxter operators given in Example \ref{ex:anti
operator}, $R_{1}, R_{2}, R_{5}, R_{6}$, $R_{7}$, $R_{8}$ are
admissible. Furthermore, the corresponding admissible Novikov
algebras given by Eq.~(\ref{eq:cons}) are also listed in Example~
\ref{ex:anti operator}. Note that $(A4)_1$ in
Proposition~\ref{thm:classification} is the only non-commutative
admissible Novikov algebra obtained by this way in the sense of
isomorphisms.



\end{ex}

\begin{cor}\label{cor:admissible anti O-operator}
Let $(\frak g,[-,-])$ be a Lie algebra. Then there is a compatible
admissible Novikov algebra structure on $\frak g$ if and only if
there exists an invertible admissible anti-$\mathcal{O}$-operator
of $(\frak g,[-,-])$.
\end{cor}

\begin{proof}
Let $x,y,z\in A$. Suppose that $(\frak g,\circ)$ is a compatible
admissible Novikov algebra structure on $(\frak g,[-,-])$. Then by
Eq.~(\ref{eq:defi:admissible Novikov algebras1}), we have
$$2\mathcal{L}_{\circ}(x)\mathcal{L}_{\circ}(y)z-2\mathcal{L}_{\circ}(x)\mathcal{L}_{\circ}(z)y=\mathcal{L}_{\circ}(\mathcal{L}_{\circ}(x)y)z-\mathcal{L}_{\circ}(\mathcal{L}_{\circ}(x)z)y.$$
By Corollary \ref{cor:anti O-operators and anti-pre-Lie algebras},
$T=\mathrm{id}$ is an invertible admissible
anti-$\mathcal{O}$-operator of $(\frak g,[-,-])$ associated to
$(-\mathcal{L}_{\circ},\frak g)$. Conversely, suppose that
$T:V\rightarrow \frak g$ is an invertible admissible
anti-$\mathcal{O}$-operator of $(\frak g,[-,-])$ associated to
$(\rho, V)$. Let $x\circ_{\frak g} y=-T(\rho(x)T^{-1}(y))$. Then
\begin{eqnarray*}
2x\circ_{\frak g}[y,z]
&=&2T\rho(x)\rho(y)T^{-1}(z)-2T\rho(x)\rho(z)T^{-1}(y)\\
&=&T(\rho(T(\rho(x)T^{-1}(y)))T^{-1}(z))-T(\rho(T(\rho(x)T^{-1}(z)))T^{-1}(y))\\
&=&(x\circ_{\frak g} y)\circ_{\frak g} z-(x\circ_{\frak g} z)\circ_{\frak g} y.
\end{eqnarray*}
By Corollary \ref{cor:anti O-operators and anti-pre-Lie algebras},
$(\frak g,\circ_{\frak g})$ is a compatible admissible Novikov algebra
structure on $(\frak g,[-,-])$.
\end{proof}

\begin{pro}
Let $\mathcal{B}$ be a nondegenerate commutative 2-cocycle on a
Lie algebra $(\frak g,[-,-])$. Define a bilinear operation
$\circ: {\frak g}\otimes \frak g\rightarrow \frak g$ by Eq.~(\ref{eq:thm:commutative
2-cocycles and anti-pre-Lie algebras}). Then $(\frak g,\circ)$ is an
admissible Novikov algebra if and only if the following equation
holds:
\begin{equation}\label{eq:addition}
2\mathcal{B}([x,w],[y,z])=\mathcal{B}([x\circ y,
w],z)-\mathcal{B}([x\circ z,w],y)),\;\;\forall x,y,z,w\in \frak g.
\end{equation}
\end{pro}

\begin{proof}
By Theorem \ref{thm:commutative 2-cocycles and anti-pre-Lie
algebras}, $(\frak g,\circ)$ is an anti-pre-Lie algebra. Moreover, for
all $x,y,z,w\in \frak g$,
\begin{eqnarray*}
2\mathcal{B}([x,w],[y,z])+\mathcal{B}(y,[x\circ
z,w])-\mathcal{B}(z,[x\circ y,w])
=\mathcal{B}(w,2x\circ[y,z]-(x\circ z)\circ y+(x\circ y)\circ z).
\end{eqnarray*}
Therefore $(\frak g,\circ)$ is an admissible Novikov algebra if and only
if Eq.~(\ref{eq:addition}) holds.
\end{proof}

\begin{rmk} The above conclusion can be obtained from another
approach. Let $T:\frak g\rightarrow \frak g^{*}$ be a linear map
defined by Eq.~(\ref{eq:thm:commutative 2-cocycles and
anti-pre-Lie algebras2}). Then it is straightforward to show that
$T^{-1}$ is an admissible anti-$\mathcal{O}$-operator of $(\frak
g,[-,-])$ associated to $(\mathrm{ad}^{*},\frak g^{*})$ if and
only if Eq.~(\ref{eq:addition}) holds. Thus the same conclusion
follows.
\end{rmk}



\begin{ex}
Let $(A,\circ)$ be an anti-pre-Lie algebra in ~Proposition
\ref{ex:anti-pre-Lie algebras from symmetric bilinear forms}, or
equivalently, one of the anti-pre-Lie algebras given in
Proposition~\ref{pro:classification}. Then $(A,\circ)$ is an
admissible Novikov algebra if and only if $(A,\circ)$ belongs to
the following cases: $(B1)$, $(B2)_k$ ($k\in \{2,\cdots,n\}$) and
$(B5)$ with $\mathrm{dim}A=2$. In particular, when $\mathcal B$ is
nondegenerate, that is, $\mathcal B$ is a nondegenerate
commutative 2-cocycle on the sub-adjacent Lie algebra, $(A,\circ)$
is an admissible Novikov algebra if and only if $(A,\circ)$
belongs to the following cases: $(B1)$ with $\dim A=1$ and $(B5)$
with $\mathrm{dim}A=2$ which is isomorphic to $(A4)_0$ in
Proposition~\ref{thm:classification}.
\end{ex}


\subsection{Constructions from commutative associative algebras}

We introduce the notion of admissible pairs on commutative
associative algebras as a generalization  of derivations. Then we
give the constructions of Novikov algebras and the corresponding
admissible Novikov algebras from commutative associative algebras
with admissible pairs, generalizing the known various
constructions of Novikov algebras from commutative associative
algebras with derivations. Especially under certain conditions,
they coincide with the admissible Novikov algebra structures
induced from the natural nondegenerate  commutative 2-cocycles on
the Lie algebras obtained from the commutative associative
algebras with derivations.

\begin{defi}
Let $(A,\cdot)$ be a commutative associative algebra. {A pair of
linear maps $P,Q:A\rightarrow A$ is called \textbf{admissible}}
if
\begin{equation}\label{eq:defi:admissible pair}
Q(x \cdot y)=Q(x)\cdot y+x\cdot P(y), \;\;\forall x,y\in A.
\end{equation}
We denote it by $(P,Q)$.
\end{defi}

\begin{ex}
Let $(A,\cdot)$ be a commutative associative algebra.
\begin{enumerate}
\item Let $P$ be a derivation on $(A,\cdot)$, that is, $(A,P)$ is
a commutative differential algebra (\cite{GK}).
\begin{enumerate}
\item $(P,P)$ is an admissible pair. \item $(P,P+\lambda {\rm
Id})$ is an admissible pair, where $\lambda\in \mathbb F$. \item
$(P,P+\mathcal L_\cdot (a))$ is an admissible pair, where $a\in
A$. \item Let $Q:A\rightarrow A$ be a linear map. Recall
(\cite{LLB}) that $Q$ is called {\bf admissible to $(A,P)$} if
$$Q(x)\cdot y=Q(x\cdot y)+x\cdot P(y),\;\;\forall x,y\in A.$$
This notion was introduced to construct a suitable dual
representation of a commutative differential algebra. 
Then $Q$ is admissible to $(A,P)$ if and only if $(P,-Q)$ is an
admissible pair. \item Let $\mathcal B$ be a nondegenerate
symmetric invariant bilinear form on $(A,\cdot)$, that is,
$(A,\cdot,\mathcal B)$ is a symmetric (commutative) Frobenius
algebra.  Let $\hat P:A\rightarrow A$ be the adjoint operator of
$P$ with respect to $\mathcal B$, that is,
$$\mathcal B(P(x),y)=\mathcal B(x,\hat P(y)),\forall x,y\in A.$$
Then $\hat P$ is admissible to $(A,P)$ (\cite{LLB}), or
equivalently, $(P,-\hat P)$ is an admissible pair.
\end{enumerate}
\item Let $Q:A\rightarrow A$ be a linear map satisfying
$$Q(x\cdot y)=Q(x)\cdot y+x\cdot Q(y)+\lambda x\cdot y,\;\;\forall
x,y\in A,$$ where $\lambda\in \mathbb F$. Then $(Q+\lambda {\rm
Id},Q)$ is an admissible pair. \item Let $Q:A\rightarrow A$ be a
linear map satisfying
$$Q(x\cdot y)=Q(x)\cdot y+x\cdot Q(y)+a\cdot x\cdot y,\;\;\forall
x,y\in A,$$ where $a\in A$. Then $(Q+\mathcal L_\cdot (a),Q)$ is
an admissible pair.
\end{enumerate}

\end{ex}

The admissible pairs on commutative associative algebras are
related to commutative 2-cocycles on Lie algebras.

\begin{pro}\label{ex:commutative 2-cocycle}
Let $(P,Q)$ be an admissible pair on a commutative associative
algebra $(A,\cdot)$. Define a bilinear operation $[-,-]:A\otimes
A\rightarrow A$ by
\begin{equation}\label{eq:ex:commutative 2-cocycle}
[x,y]=Q(x)\cdot y-x\cdot Q(y)=P(x)\cdot y-x\cdot P(y), \forall x,y\in A.
\end{equation}
Then $(A,[-,-])$ is a Lie algebra. Moreover, if there is a
symmetric invariant bilinear form $\mathcal{B}$ on $(A,\cdot)$ 
 in the sense of  Eq.~(\ref{eq:inv-asso}),
then $\mathcal{B}$ is a commutative 2-cocycle on $(A,[-,-])$.
\end{pro}

\begin{proof}
The first half part is obtained by a direct proof or follows from
Proposition~\ref{ex:Novikov algebra from admissible pair} or
Corollary~\ref{cor:admissible Novikov algebra from admissible
pair} since $(A,[-,-])$ is the sub-adjacent Lie algebra of an
admissible Novikov algebra. Note that the second equality in Eq.~(\ref{eq:ex:commutative 2-cocycle}) is due to
Eq.~(\ref{eq:defi:admissible pair}) and the fact that $(A,\cdot)$ is commutative.
For the second half part, let
$x,y,z\in A$. Then we have
\begin{eqnarray*}
&\mathcal{B}([x,y],z)=\mathcal{B}(Q(x)\cdot y,z)-\mathcal{B}(x\cdot Q(y),z)=\mathcal{B}(Q(x),y\cdot z)-\mathcal{B}(Q(y),x\cdot z),\\
&\mathcal{B}([y,z],x)=\mathcal{B}(Q(y)\cdot z,x)-\mathcal{B}(y\cdot Q(z),x)=\mathcal{B}(Q(y),z\cdot x)-\mathcal{B}(Q(z),y\cdot x),\\
&\mathcal{B}([z,x],y)=\mathcal{B}(Q(z)\cdot
x,y)-\mathcal{B}(z\cdot Q(x),y)=\mathcal{B}(Q(z),x\cdot
y)-\mathcal{B}(Q(x),z\cdot y).
\end{eqnarray*}
Thus $\mathcal{B}$ is a commutative 2-cocycle on $(A,[-,-])$.
\end{proof}

In particular, as a special case, we get the following conclusion.

\begin{cor}\label{cor:construction}
Let $P$ be a derivation on a commutative associative algebra
$(A,\cdot)$. Then Eq.~(\ref{eq:ex:commutative 2-cocycle}) defines
a Lie algebra $(A,[-,-])$. If $\mathcal B$ is a symmetric
invariant bilinear form on $(A,\cdot)$, then $\mathcal B$ is a
commutative 2-cocycle on $(A,[-,-])$. In particular, if
$(A,\cdot,\mathcal B)$ is a 
symmetric (commutative) Frobenius
algebra, 
then $\mathcal B$ is a
nondegenerate commutative 2-cocycle on the Lie algebra $(A,[-,-])$.
\end{cor}


\begin{pro}\label{ex:Novikov algebra from admissible pair}
Let $(P,Q)$ be an admissible pair on a commutative associative
algebra $(A,\cdot)$. Define a bilinear operation $\star$ on $A$ by
\begin{equation}\label{eq:ex:Novikov algebra from admissible pair}
x\star y=x\cdot Q(y), \;\;\forall x,y\in A.
\end{equation}
Then $(A,\star)$ is a Novikov algebra.
\end{pro}

\begin{proof}
Let $x,y,z\in A$. Then we have
\begin{eqnarray*}
(x\star y)\star z-x\star(y\star z)&=&x\cdot Q(y)\cdot Q(z)-x\cdot
Q(y\cdot Q(z))=-x\cdot y\cdot P(Q(z))\\&=&(y\star x)\star z-y\star
(x\star z),\\
(x\star y)\star z&=&x\cdot Q(y)\cdot Q(z)=(x\star z)\star y.
\end{eqnarray*}
Hence $(A,\star)$ is a Novikov algebra.
\end{proof}

\begin{ex}
Let $P$ be a derivation on a commutative associative algebra
$(A,\cdot)$. There is the following known construction of Novikov
algebras.
Define a bilinear operation $\star:A\otimes A\rightarrow A$ by
\begin{equation}x\star y=x\cdot P(y)+a\cdot x\cdot y,\;\; \forall x,y\in A.\end{equation}
Then we have the following results.
\begin{enumerate}
\item  (\cite{Gel}, S. Gelfand) $a=0$ makes $(A,\circ)$ a Novikov
algebra, corresponding to the fact that $(P,P)$ is an admissible
pair. \item (\cite{Fil1989}) $a\in\mathbb{F}$  makes $(A,\circ)$ a
Novikov algebra, corresponding to the fact that $(P,P+a {\rm Id})$
is an admissible pair. \item (\cite{Xu1996}) $a\in A$ makes
$(A,\circ)$ a Novikov algebra, corresponding to the fact that
$(P,P+\mathcal L_\cdot(a))$ is an admissible pair.
\end{enumerate}
\end{ex}


\begin{cor}\label{cor:admissible Novikov algebra from admissible pair}
Let $(P,Q)$ be an admissible pair on a commutative associative
algebra $(A,\cdot)$. Define a bilinear operation $\circ$ on $A$ by
\begin{equation}\label{eq:cor:admissible Novikov algebra from admissible pair}
x\circ y=x\cdot Q(y)+2Q(x)\cdot y,\;\;\forall x,y\in A.
\end{equation}
Then $(A,\circ)$ is an admissible Novikov algebra. Moreover the
sub-adjacent Lie algebra $(A,[-,-])$ of $(A,\circ)$ satisfies
Eq.~(\ref{eq:ex:commutative 2-cocycle}).
\end{cor}


\begin{ex}\label{ex:der}
Let $P$ be a derivation on a commutative associative algebra
$(A,\cdot)$. Define a bilinear operation $\circ: A\otimes
A\rightarrow A$ by
\begin{equation}
x\circ y=x\cdot P(y)+2P(x)\cdot y+a\cdot x\cdot y,\;\;\forall
x,y\in A, \end{equation} where $a\in {\mathbb F}$ or $a\in A$.
Then $(A,\circ)$ is an admissible Novikov algebra. Note that in
this case, both $(P,P+\frac{1}{3}a {\rm Id})$ ($a\in \mathbb F$)
and $(P,P+\frac{1}{3}\mathcal L_\cdot(a))$ ($a\in A$) are
admissible pairs.
\end{ex}


\begin{cor}\label{cor:invariant}
Let $Q$ be a derivation on a commutative associative algebra
$(A,\cdot)$. Suppose that $\mathcal B$ is a nondegenerate
symmetric invariant bilinear form on $(A,\cdot)$ and $\hat Q$ is
the adjoint operator of $Q$ with respect to $\mathcal B$. Then the following conditions are equivalent.
\begin{enumerate}
\item \label{it:it1}
The
admissible Novikov algebra $(A,\circ)$ defined by
Eq.~(\ref{eq:cor:admissible Novikov algebra from admissible pair})
is exactly the one defined by Eq.~(\ref{eq:thm:commutative
2-cocycles and anti-pre-Lie algebras}) through the nondegenerate
commutative 2-cocycle $\mathcal B$ on the Lie algebra $(A,[-,-])$
given by Eq.~(\ref{eq:ex:commutative 2-cocycle}).
\item \label{it:it2} $\mathcal B$ is invariant on the admissible Novikov algebra
$(A,\circ)$ defined by Eq.~(\ref{eq:cor:admissible Novikov algebra
from admissible pair}).
\item \label{it:it3}The following equation
holds:
\begin{equation}\label{eq:condd}
x\cdot Q(y)=-x\cdot \hat Q(y),\forall x,y\in A.
\end{equation}
\end{enumerate}
In particular, if $\hat Q=-Q$, that is, $Q$ is skew-self-adjoint
with respect to $\mathcal B$, then $\mathcal B$ is invariant on
the admissible Novikov algebra $(A,\circ)$ defined by
Eq.~(\ref{eq:cor:admissible Novikov algebra from admissible
pair}).
\end{cor}

\begin{proof}
(\ref{it:it1})$\Longleftrightarrow$ (\ref{it:it2}). It follows from Corollary~\ref{cor:corres}.

(\ref{it:it2})$\Longleftrightarrow$ (\ref{it:it3}).
Let $x,y,z\in A$. Then we have
\begin{eqnarray*}
\mathcal B(x\circ y,z)-\mathcal B(y,[x,z]) &=&\mathcal B(x\cdot
Q(y)+2y\cdot Q(x), z)-\mathcal B(y,Q(x)\cdot z-x\cdot Q(z))\\
&=&\mathcal B(y,Q(x)\cdot z+x\cdot Q(z))+\mathcal B(y,\hat Q(x\cdot z))\\
&=&\mathcal B(y,Q(x)\cdot z+x\cdot Q(z))+\mathcal B(y, -x\cdot
Q(z)+\hat Q(x)\cdot z))\\
&=&\mathcal B(y,z\cdot Q(x)+z\cdot \hat Q(x)).
\end{eqnarray*}
Hence $\mathcal B$ is invariant on $(A,\circ)$ if and only if
Eq.~(\ref{eq:condd}) holds. The rest follows immediately.
\end{proof}

\begin{rmk}\label{rmk:cons} With the conditions in Corollary~\ref{cor:invariant},
$(Q,-\hat Q)$ is an admissible pair. Then by Example~\ref{ex:der}
and Corollary~\ref{cor:admissible Novikov algebra from admissible
pair} respectively, there are two admissible Novikov algebras as
follows.
\begin{equation*}\label{rmk:admiisible Novikov algebras from
derivations2} x\circ_{1}y=x\cdot Q(y)+2Q(x)\cdot y,\;\; \forall
x,y\in A,
\end{equation*}
\begin{equation*}\label{rmk:admiisible Novikov algebras from derivations3}
x\circ_{2}y=-x\cdot \hat Q(y)-2 \hat Q(x)\cdot y,\;\; \forall
x,y\in A.
\end{equation*}
On the other hand, there is still a compatible anti-pre-Lie
algebra structure $(A,\circ_3)$ defined by
Eq.~(\ref{eq:thm:commutative 2-cocycles and anti-pre-Lie
algebras}) on the Lie algebra $(A,[-,-])$ given by
Eq.~(\ref{eq:ex:commutative 2-cocycle}) as follows.
\begin{equation}\label{rmk:admiisible Novikov algebras from derivations4}
x\circ_3 y=Q(x)\cdot y+x\cdot Q(y)-\hat Q(x)\cdot y, \;\;\forall
x,y\in A.
\end{equation}
Moreover, in general $(A,\circ_3)$ is only an anti-pre-Lie
algebra, and it is an admissible Novikov algebra if and only if
$$x\cdot(Q\hat Q-\hat Q Q)(y)\cdot z=x\cdot y\cdot(Q\hat Q-\hat Q Q)(z), \;\;\forall x,y,z\in A. $$
Note that $(A,\circ_1)$, $(A,\circ_2)$ and $(A,\circ_3)$ are the
same if and only if Eq.~(\ref{eq:condd}) holds, including the
conclusion given in Corollary~\ref{cor:invariant}.
\end{rmk}


\begin{ex}\label{ex:admiisible Novikov algebras from derivations4}
Let $(A,\cdot)$ be the 2-dimensional complex commutative
associative algebra with a basis $\{e_{1}, e_{2}\}$ whose non-zero
products are given by
$$e_{1}\cdot e_{1}=e_{1}, e_{1}\cdot
e_{2}=e_{2}.$$ Then there is a nondegenerate symmetric invariant
bilinear form $\mathcal{B}$ on $(A,\cdot)$ where the non-zero
values are given by $\mathcal{B}(e_{1}, e_{2})=1$. Let
$Q:A\rightarrow A$ be a linear map given by $Q(e_{1})=0,
Q(e_{2})=e_{2}$. Then $Q$ is a derivation of $(A, \cdot)$ and the
adjoint
operator $\hat Q$ 
is given by $\hat Q(e_{1})=e_{1}, \hat Q(e_{2})=0$. The three
anti-pre-Lie algebras $(A,\circ_1)$, $(A,\circ_2)$ and
$(A,\circ_3)$ given in Remark~\ref{rmk:cons} are different admissible
Novikov algebras, which are isomorphic to $(A4)_1$, $(A4)_{-2}$
and $(A4)_0$ in Proposition~\ref{thm:classification} respectively.
\end{ex}

At the end of this subsection, we extend the construction of Lie
algebras from Novikov algebras or admissible Novikov algebras
given by Proposition~\ref{lem:BN} and Corollary~\ref{cor:BN-A}.

\begin{pro}\label{pro:from Novikov to Lie}
Let $(A,\cdot)$ be a commutative associative algebra with an
admissible pair $(P,Q)$. Let $(V,\star)$ be a Novikov algebra.
Define a bilinear operation $[-,-]$ on $A\otimes V$ by
\begin{equation}\label{eq:pro:from Novikov to Lie1}
[x\otimes u, y\otimes v]=Q(x)\cdot y\otimes u\star v-x\cdot
Q(y)\otimes v\star u, \;\;\forall x,y\in A, u,v\in V.
\end{equation}
Then $(A\otimes V, [-,-])$ is a Lie algebra.
\end{pro}
\begin{proof}
The operation $[-,-]$ is obviously skew-symmetric. For all
$x,y,z\in A, u,v,w\in V$, we have
\begin{eqnarray*}
&&[[x\otimes u,y\otimes v],z\otimes w]+[[y\otimes v, z\otimes w],x\otimes u]+[[z\otimes w, x\otimes u], y\otimes v]\\
&&=Q(x)\cdot Q(y)\cdot z\otimes((u\star v)\star w-(v\star u)\star w-u\star(v\star w)+v\star(u\star w))\\
&&\hspace{0.4cm}+Q(y)\cdot Q(z)\cdot x\otimes((v\star w)\star u-(w\star v)\star u-v\star(w\star u)+w\star(v\star u))\\
&&\hspace{0.4cm}+Q(z)\cdot Q(x)\cdot y\otimes((w\star u)\star v-(u\star w)\star v-w\star(u\star v)+u\star(w\star v))\\
&&\hspace{0.4cm}+x\cdot y\cdot P(Q(z))\otimes((w\star u)\star v-(w\star v)\star u)+y\cdot z\cdot P(Q(x))\otimes((u\star v)\star w-(u\star w)\star v)\\
&&\hspace{0.4cm}+ z\cdot x\cdot P(Q(y))\otimes((v\star w)\star
u-(v\star u)\star w)=0.
\end{eqnarray*}
Hence  $(A\otimes V, [-,-])$ is a Lie algebra.
\end{proof}

\begin{cor}\label{pro:from admissible Novikov to Lie}
Let $(A,\cdot)$ be a commutative associative algebra with an
admissible pair $(P,Q)$. Let $(V,\circ)$ be an admissible Novikov
algebra. Define a bilinear operation $[-,-]$ on $A\otimes V$ by
\begin{equation}\label{eq:pro:from admissible Novikov to Lie1}
[x\otimes u, y\otimes v]=(Q(x)\cdot y+2x\cdot Q(y))\otimes u\circ
v-(x\cdot Q(y)+2Q(x)\cdot y)\otimes v\circ u, \forall x,y\in A,
u,v\in V.
\end{equation}
Then $(A\otimes V, [-,-])$ is a Lie algebra.
\end{cor}
\begin{proof}
It follows from Propositions~\ref{pro:from Novikov to Lie}
and~\ref{pro:admissible Novikov algebras and Novikov algebras}.
\end{proof}

\subsection{Anti-pre-Lie Poisson algebras and Novikov-Poisson algebras}

We extend the correspondence between Novikov algebras and
admissible Novikov algebras to the level of Poisson type
structures, and hence introduce the notions of anti-pre-Lie Poisson
algebras and admissible Novikov-Poisson algebras. The relationships with transposed Poisson algebras and
Novikov-Poisson algebras as well as a tensor theory are given.

\begin{defi}\label{defi:anti-pre-Lie Poisson}
An \textbf{anti-pre-Lie Poisson algebra} is a triple
$(A,\cdot,\circ)$, where $(A,\cdot)$ is a commutative
associative algebra and $(A,\circ)$ is an anti-pre-Lie algebra satisfying the following conditions:
\begin{equation}\label{eq:defi:anti-pre-Lie Poisson1}
2(x\circ y)\cdot z-2(y\circ x)\cdot z=y\cdot(x\circ z)-x\cdot(y\circ z),
\end{equation}
\begin{equation}\label{eq:defi:anti-pre-Lie Poisson2}
2x\circ(y\cdot z)=(z\cdot x)\circ y+z\cdot(x\circ y),
\end{equation}
for all $x,y,z\in A$. An \textbf{admissible Novikov-Poisson algebra} is an anti-pre-Lie
Poisson algebra $(A,\cdot,\circ)$ such that $(A,\circ)$ is
an admissible Novikov algebra.
\end{defi}

\begin{ex} Obviously, for any anti-pre-Lie algebra $(A,\circ)$,
$(A,\cdot,\circ)$ is an anti-pre-Lie Poisson algebra with $(A,
\cdot)$ being trivial. In particular, by a straightforward
computation, any 2-dimensional complex anti-pre-Lie Poisson
algebra $(A,\cdot,\circ)$ is an admissible Novikov-Poisson algebra
or the one with $(A,\cdot)$ being trivial. On the other hand, let
$A$ be a 3-dimensional vector space with a basis $\lbrace e_{1},
e_{2}, e_{3}\rbrace$. Let $\cdot,\circ:A\otimes A\rightarrow A$ be
two bilinear operations given by the following non-zero products
respectively:
$$e_{3}\cdot e_{3}=e_{3};\;\;\;\;\;
e_{1}\circ e_{2} =e_{2}, e_{3}\circ e_{3}=e_{3}.$$ Then
$(A,\cdot,\circ)$ is an anti-pre-Lie Poisson algebra, which is not
an admissible Novikov-Poisson algebra.
\end{ex}




\begin{lem} \label{pro:anti-pre-Lie Poisson}
Let $(A,\cdot)$ be a commutative associative algebra, and
$\circ:A\otimes A\rightarrow A$ be a bilinear operation. Suppose that
Eqs.~(\ref{eq:defi:anti-pre-Lie Poisson1}) and
(\ref{eq:defi:anti-pre-Lie Poisson2}) hold.
Then
the following equations hold:
\begin{eqnarray}\label{eq:defi:anti-pre-Lie Poisson3}
&&2x\circ(y\cdot z)=(y\cdot x)\circ z+y\cdot(x\circ z),\\
\label{eq:defi:anti-pre-Lie Poisson4}
&&(z\cdot x)\circ y+z\cdot(x\circ y)=(y\cdot x)\circ z+y\cdot(x\circ z),\\
\label{eq:defi:anti-pre-Lie Poisson5}
&&x\cdot(y\circ z)-y\cdot(x\circ z)=2y\circ(x\cdot z)-2x\circ(y\cdot z),\\
\label{eq:defi:anti-pre-Lie Poisson6}
&&x\cdot(y\circ z)-y\cdot(x\circ z)=2x\cdot(z\circ y)-2y\cdot(z\circ x),\\
\label{eq:defi:anti-pre-Lie Poisson9}
&&x\cdot(z\circ y)-y\cdot(z\circ x)=y\circ(x\cdot z)-x\circ(y\cdot z),\\
\label{eq:defi:anti-pre-Lie Poisson7}
&&x\circ(y\cdot z)+(y\cdot z)\circ x=y\circ(x\cdot z)+(x\cdot z)\circ y,\\
\label{eq:defi:anti-pre-Lie Poisson8}
&&z\cdot(x\circ y)-z\cdot(y\circ x)=x\circ(y\cdot z)-y\circ(x\cdot z),
\end{eqnarray}
for all $x,y,z\in A$.
\end{lem}

\begin{proof}
Let $x,y,z\in A$. Then we have
\begin{enumerate}
\item Eq.~(\ref{eq:defi:anti-pre-Lie Poisson3}) holds by exchanging $y$ and $z$ in Eq.~(\ref{eq:defi:anti-pre-Lie Poisson2}).

\item Eq.~(\ref{eq:defi:anti-pre-Lie Poisson4}) follows from Eqs.~(\ref{eq:defi:anti-pre-Lie Poisson2}) and (\ref{eq:defi:anti-pre-Lie Poisson3}).

\item Eq.~(\ref{eq:defi:anti-pre-Lie Poisson5}) follows from Eq.~(\ref{eq:defi:anti-pre-Lie Poisson3}).

\item Eq.~(\ref{eq:defi:anti-pre-Lie Poisson6}) holds since
\begin{eqnarray*}
&&y\cdot(x\circ z)-x\cdot(y\circ z)\overset{(\ref{eq:defi:anti-pre-Lie Poisson1})}{=}2(x\circ y)\cdot z-2(y\circ x)\cdot z\\
&&\overset{(\ref{eq:defi:anti-pre-Lie Poisson1})}{=}4y\cdot(x\circ z)-4y\cdot(z\circ x)+2x\cdot(z\circ y)-4x\cdot(y\circ z)+4x\cdot(z\circ y)-2y\cdot(z\circ x)\\
&&=4y\cdot(x\circ z)-4x\cdot(y\circ z)-6y\cdot(z\circ x)+6x\cdot(z\circ y).
\end{eqnarray*}

\item
Eq.~(\ref{eq:defi:anti-pre-Lie Poisson9}) follows from
Eqs.~(\ref{eq:defi:anti-pre-Lie Poisson5}) and
(\ref{eq:defi:anti-pre-Lie Poisson6}).

\item Eq.~(\ref{eq:defi:anti-pre-Lie
Poisson7}) holds  since
$$(y\cdot z)\circ x-(z\cdot x)\circ y\overset{(\ref{eq:defi:anti-pre-Lie Poisson4})}{=}x\cdot(z\circ y)-y\cdot(z\circ x)
\overset{(\ref{eq:defi:anti-pre-Lie Poisson9})}{=}y\circ(x\cdot
z)-x\circ(y\cdot z). $$
\item Eq.~(\ref{eq:defi:anti-pre-Lie Poisson8}) holds since
\begin{eqnarray*}
&&x\circ(y\cdot z)-y\circ(z\cdot x)-z\cdot(x\circ y)+z\cdot(y\circ x)\\
&&\overset{(\ref{eq:defi:anti-pre-Lie Poisson3})}{=}x\circ(y\cdot z)-y\circ(z\cdot x)+(z\cdot x)\circ y-2x\circ(z\cdot y)-(z\cdot y)\circ x+2y\circ(z\cdot x)\\
&&=y\circ(z\cdot x)+(z\cdot x)\circ y-x\circ(y\cdot z)-(z\cdot y)\circ x\overset{(\ref{eq:defi:anti-pre-Lie Poisson7})}{=}0.
\end{eqnarray*}
\end{enumerate}
Thus the conclusion holds.
\end{proof}

\begin{defi}
(\cite{Xu1997})\label{defi:Novikov Poisson algebra}
A \textbf{Novikov-Poisson algebra} is a triple $(A,\cdot,\star)$, where $(A,\cdot)$ is a commutative associative algebra and $(A,\star)$ is a Novikov algebra satisfying the following conditions:
\begin{equation}\label{eq:defi:Novikov Poisson algebra1}
(x\cdot y)\star z=x\cdot(y\star z),
\end{equation}
\begin{equation}\label{eq:defi:Novikov Poisson algebra2}
(x\star y)\cdot z-(y\star x)\cdot z=x\star(y\cdot z)-y\star(x\cdot z),
\end{equation}
for all $x,y,z\in A$.
\end{defi}

The following conclusion extends the correspondence between
Novikov algebras and admissible Novikov algebras to the level of
Poisson type structures.



\begin{pro}\label{pro:from admissible Novikov Poisson to Novikov Poisson}
Let $(A,\cdot)$ be a commutative associative algebra and $\circ:A\otimes A\rightarrow A$ be a bilinear operation.
Let $(A,\star)$ be the
$(-2)$-algebra of $(A,\circ)$. Then
Eqs.~(\ref{eq:defi:anti-pre-Lie Poisson1}) and
(\ref{eq:defi:anti-pre-Lie Poisson2}) hold if and only if
Eqs.~(\ref{eq:defi:Novikov Poisson algebra1}) and
(\ref{eq:defi:Novikov Poisson algebra2}) hold. In particular,
$(A,\cdot,\circ)$ is an admissible Novikov-Poisson algebra if and only if
$(A,\cdot,\star)$ is a Novikov-Poisson algebra.
\end{pro}

\begin{proof}
{Suppose that Eqs.~(\ref{eq:defi:anti-pre-Lie Poisson1}) and
    (\ref{eq:defi:anti-pre-Lie Poisson2}) hold. Then for all $x,y,z\in A$, we have}
\begin{eqnarray*}
&&(x\cdot y)\star z-x\cdot(y\star z)\\&&=(x\cdot y)\circ z-2z\circ(x\cdot y)-x\cdot(y\circ z)+2x\cdot(z\circ y)\\
&&\overset{(\ref{eq:defi:anti-pre-Lie Poisson2})}{=}
2y\circ(z\cdot x)-2z\circ(x\cdot y)-2x\cdot(y\circ z)+2x\cdot(z\circ y)\\
&&\overset{(\ref{eq:defi:anti-pre-Lie Poisson1})}{=}
2y\circ(z\cdot x)-2z\circ(x\cdot y)+y\cdot(z\circ x)-z\cdot(y\circ x)\\
&&\overset{(\ref{eq:defi:anti-pre-Lie Poisson3})}{=}
(z\cdot y)\circ x+z\cdot(y\circ x)-(y\cdot z)\circ x-y\cdot(z\circ x)+y\cdot(z\circ x)-z\cdot(y\circ x)=0,\\
&&2(x\star y)\cdot z-2(y\star x)\cdot z+2y\star(x\cdot z)-2x\star(y\cdot z)\\
&&=6(x\circ y)\cdot z-6(y\circ x)\cdot z+2y\circ(x\cdot z)-4(x\cdot z)\circ y-2x\circ(y\cdot z)+4(y\cdot z)\circ x\\
&&\overset{(\ref{eq:defi:anti-pre-Lie
Poisson1}),(\ref{eq:defi:anti-pre-Lie Poisson3})}{=}
3y\cdot(x\circ z)-3x\cdot(y\circ z)+(x\cdot y)\circ z+x\cdot(y\circ z)-4(x\cdot z)\circ y-(y\cdot x)\circ z\\
&&\hspace{0.4cm}-y\cdot(x\circ z)+4(y\cdot z)\circ x\\
&&=2y\cdot(x\circ z)-2x\cdot(y\circ z)+4(y\cdot z)\circ x-4(x\cdot z)\circ y\\
&&\overset{(\ref{eq:defi:anti-pre-Lie Poisson5})}{=}
4x\circ (y\cdot z)-4y\circ(x\cdot z)+4(y\cdot z)\circ x-4(x\cdot z)\circ y\overset{(\ref{eq:defi:anti-pre-Lie Poisson7})}{=}0.
\end{eqnarray*}
Hence Eqs.~(\ref{eq:defi:Novikov Poisson algebra1}) and
(\ref{eq:defi:Novikov Poisson algebra2}) hold. Conversely, suppose
that Eqs.~(\ref{eq:defi:Novikov Poisson algebra1}) and
(\ref{eq:defi:Novikov Poisson algebra2}) hold. Then a similar
argument gives Eqs.~(\ref{eq:defi:anti-pre-Lie Poisson1}) and
(\ref{eq:defi:anti-pre-Lie Poisson2}). \delete{
``$\Longleftarrow$". Let $x,y,z\in A$. We have
\begin{eqnarray*}
&&2(x\circ y)\cdot z-2(y\circ x)\cdot z+x\cdot(y\circ z)-y\cdot(x\circ z)\\
&&=-\frac{1}{3}(2(x\star y)\cdot z+4(y\star x)\cdot z-2(y\star x)\cdot z-4(x\star y)\cdot z+x\cdot(y\star z)+2x\cdot(z\star y)\\
&&\hspace{0.4cm}-y\cdot(x\star z)-2y\cdot(z\star x))\\
&&\overset{(\ref{eq:defi:Novikov Poisson algebra1})}{=}
-\frac{1}{3}(2(y\star x)\cdot z-2(x\star y)\cdot z+2x\cdot(z\star y)-2y\cdot(z\star x))\\
&&\overset{(\ref{eq:defi:Novikov Poisson algebra1})}{=}
-\frac{1}{3}(2(y\star x)\cdot z-2(x\star y)\cdot z+2z\cdot(x\star y)-2z\cdot(y\star x))=0,\\
&&2x\circ(y\cdot z)-(z\cdot x)\circ y-z\cdot(x\circ y)\\
&&=-\frac{1}{3}(2x\star(y\cdot z)+4(y\cdot z)\star x-(z\cdot x)\star y-2y\star(z\cdot x)-z\cdot(x\star y)-2z\cdot(y\star x))\\
&&\overset{(\ref{eq:defi:Novikov Poisson algebra1})}{=}
-\frac{1}{3}(2x\star(y\cdot z)+2z\cdot(y\star x)-2z\cdot(x\star y)-2y\star(z\cdot x))\overset{(\ref{eq:defi:Novikov Poisson algebra2})}{=}0.
\end{eqnarray*}
Hence Eqs.~(\ref{eq:defi:anti-pre-Lie Poisson1}) and
(\ref{eq:defi:anti-pre-Lie Poisson2}) hold.}
\end{proof}



\begin{pro}
Let $(P,Q)$ be an admissible pair on a commutative associative
algebra $(A,\cdot)$.
\begin{enumerate}
\item \label{it:1} Let $(A,\star)$ be the Novikov algebra given
by Eq.~(\ref{eq:ex:Novikov algebra from admissible pair}) in
Proposition \ref{ex:Novikov algebra from admissible pair}. Then
$(A,\cdot,\star)$ is a Novikov-Poisson algebra.
\item \label{it:2} Let $(A,\circ)$ be the admissible Novikov
algebra given by Eq.~(\ref{eq:cor:admissible Novikov algebra from
admissible pair}) in Corollary \ref{cor:admissible Novikov algebra
from admissible pair}. Then $(A,\cdot,\circ)$ is an admissible
Novikov-Poisson algebra.
\end{enumerate}
\end{pro}

\begin{proof}
(\ref{it:1}). Let $x,y,z\in A$. Then we have
\begin{eqnarray*}
(x\cdot y)\star z=x\cdot y\cdot Q(z)&=&x\cdot (y\star z), \\
x\star (y\cdot z)-y\star(x\cdot z)&=&x\cdot Q(y\cdot z)-y\cdot Q(x\cdot z)=x\cdot Q(y)\cdot z-y\cdot Q(x)\cdot z\\&=&(x\star y)\cdot z-(y\star x)\cdot z.
\end{eqnarray*}
Hence Eqs.~(\ref{eq:defi:Novikov Poisson algebra1}) and
(\ref{eq:defi:Novikov Poisson algebra2}) hold.

(\ref{it:2}). It follows from Item~(\ref{it:1}) and Proposition \ref{pro:from admissible Novikov Poisson to Novikov Poisson}.
\end{proof}

\begin{ex}
Let $P$ be a derivation on a commutative associative algebra
$(A,\cdot)$. Define two bilinear operations $\star,\circ: A\otimes
A\rightarrow A$ respectively by
\begin{equation*}
x\star y=x\cdot P(y)+a\cdot x\cdot y,\;\; \forall x,y\in A,
\end{equation*}
\begin{equation*}
x\circ y=x\cdot P(y)+2P(x)\cdot y+a\cdot x\cdot y,\;\;\forall
x,y\in A, \end{equation*} where $a\in {\mathbb F}$ or $a\in A$.
Then $(A,\cdot,\star)$ is a Novikov-Poisson algebra (\cite{Xu1997}) and
$(A,\cdot,\circ)$ is an admissible Novikov-Poisson algebra.
\end{ex}


Next we consider the structure of the sub-adjacent Lie algebras of the anti-pre-Lie algebras in anti-pre-Lie Poisson algebras.

\begin{defi}(\cite{Bai2020})\label{defi:transposed Poisson algebra}
A \textbf{transposed Poisson algebra} is a triple
$(A,\cdot,[-,-])$, where $(A,\cdot)$ is a commutative associative
algebra, and $(A,[-,-])$ is a Lie algebra satisfying the following
equation:
\begin{equation}\label{eq:defi:transposed Poisson algebra}
2z\cdot[x,y]=[z\cdot x,y]+[x,z\cdot y], \forall x,y,z\in A.
\end{equation}
\end{defi}

\begin{ex}\label{pro:from Novikov Poisson algebra to transposed Poisson algebra}
Let $(A,\cdot,\star)$ be a Novikov-Poisson algebra and $(A,[-,-])$ be the sub-adjacent Lie algebra of $(A,\star)$. Then $(A,\cdot,[-,-])$ is a transposed Poisson algebra (\cite{Bai2020}).
\end{ex}

\begin{pro}\label{pro:sub-adjacent transposed Poisson algebra}
Let $(A,\cdot)$ be a commutative associative algebra and
$(A,\circ)$ be a Lie-admissible algebra such that
Eqs.~(\ref{eq:defi:anti-pre-Lie Poisson1}) and
(\ref{eq:defi:anti-pre-Lie Poisson2}) hold.
Suppose that $(A,[-,-])$ is the
 sub-adjacent
Lie algebra of $(A,\circ)$.
Then $(A,\cdot,[-,-])$ is a transposed Poisson algebra. 
In particular, if $(A,\cdot,\circ)$ is an anti-pre-Lie Poisson
algebra, then $(A,\cdot,[-,-])$ is a transposed Poisson algebra.
\end{pro}
\begin{proof}
Let $x,y,z\in A$. Then we have
\begin{eqnarray*}
&&[z\cdot x,y]+[x,z\cdot y]-2z\cdot[x,y]\\
&&=(z\cdot x)\circ y-y\circ(z\cdot x)+x\circ(z\cdot y)-(z\cdot y)\circ x-2z\cdot[x,y]\\
&&\overset{(\ref{eq:defi:anti-pre-Lie Poisson2})}{=}2x\circ(y\cdot z)-z\cdot(x\circ y)-y\circ (z\cdot x)+x\circ(z\cdot y)-2y\circ(x\cdot z)+z\cdot(y\circ x)\\
&&\hspace{0.4cm}-2z\cdot(x\circ y)+2z\cdot(y\circ x)\\
&&=3(x\circ(y\cdot z)-y\circ (z\cdot x)-z\cdot(x\circ y)+z\cdot(y\circ x))\overset{(\ref{eq:defi:anti-pre-Lie Poisson8})}{=}0.
\end{eqnarray*}
Thus $(A,\cdot,[-,-])$ is a transposed Poisson algebra.
\end{proof}

\begin{ex}
Let $(P,Q)$ be an admissible pair on a commutative associative
algebra $(A,\cdot)$. Define a Lie algebra $(A,[-,-])$  by
Eq.~(\ref{eq:ex:commutative 2-cocycle}). Then $(A,\cdot,[-,-])$ is
a transposed Poisson algebra. Note that by Proposition~\ref{ex:commutative 2-cocycle}, if there is a
symmetric invariant bilinear form $\mathcal{B}$ on $(A,\cdot)$,
then $\mathcal{B}$ is also a commutative 2-cocycle on the Lie
algebra $(A,[-,-])$.
\end{ex}

Conversely, in the nondegenerate case, we have the following conclusion.

\begin{pro}\label{thm:bilinear form on transposed Poisson algebra}
Let $(A,\cdot,[-,-])$ be a transposed Poisson algebra. Suppose
there is a nondegenerate symmetric bilinear form $\mathcal{B}$ on
$A$, such that $\mathcal{B}$ is invariant on $(A,\cdot)$ and a
commutative 2-cocycle on $(A,[-,-])$. Then  $(A,\cdot,\circ)$ is
an anti-pre-Lie Poisson algebra, where the bilinear operation $\circ$
is given by Eq.~(\ref{eq:thm:commutative 2-cocycles and
anti-pre-Lie algebras}).
\end{pro}

\begin{proof} By Theorem~\ref{thm:commutative 2-cocycles and anti-pre-Lie algebras}, $(A,\circ)$ is an anti-pre-Lie algebra.
Let $x,y,z,w\in A$. Then we have
\begin{eqnarray*}
&&\mathcal{B}(2[x,y]\cdot z+x\cdot(y\circ z)-y\cdot(x\circ z),w)=\mathcal{B}(z,2[x,y]\cdot w+[y,x\cdot w]-[x,y\cdot w])=0,\\
&&\mathcal{B}(2x\circ(y\cdot z)-(z\cdot x)\circ y-z\cdot(x\circ y),w)=\mathcal{B}(y,2z\cdot[x,w]-[z\cdot x,w]-[x,z\cdot w])=0.
\end{eqnarray*}
Thus $(A,\cdot,\circ)$ is an anti-pre-Lie Poisson algebra.
\end{proof}

\begin{ex}
Let $Q$ be a derivation on a commutative associative algebra
$(A,\cdot)$. Suppose that $\mathcal B$ is a nondegenerate
symmetric invariant bilinear form on $(A,\cdot)$ and $\hat Q$ is
the adjoint operator of $Q$ with respect to $\mathcal B$. Define a
bilinear operation $[-,-]$ on $A$ by Eq.~(\ref{eq:ex:commutative
2-cocycle}). Then $(A,\cdot,[-,-])$ is a transposed Poisson
algebra (\cite{Bai2020}) and $\mathcal B$ is a commutative
2-cocycle on $(A,[-,-])$. Furthermore, by Remark~\ref{rmk:cons} and Proposition~\ref{thm:bilinear form on transposed Poisson algebra},
$(A,\cdot,\circ_3)$ is an anti-pre-Lie Poisson algebra, where
$(A,\circ_3)$ is defined by Eq. (\ref{rmk:admiisible Novikov
algebras from derivations4}). It is an admissible Novikov-Poisson
algebra if Eq.~(\ref{eq:condd}) holds and in this
case, $(A,\circ_3)$ is exactly defined by
Eq.~(\ref{eq:cor:admissible Novikov algebra from admissible
pair}).
\end{ex}

At the end of this subsection, we consider a tensor theory
of anti-pre-Lie Poisson algebras and admissible Novikov-Poisson
algebras.

\begin{thm}\label{thm:tensor algebra}
Let $(A_{1},\cdot_{1},\circ_{1})$ and
$(A_{2},\cdot_{2},\circ_{2})$ be two anti-pre-Lie Poisson
algebras. Define two bilinear operations $\cdot$ and $\circ$ on $A_{1}\otimes
A_{2}$ respectively by
\begin{equation}\label{eq:thm:tensor algebra1}
(x_{1}\otimes x_{2})\cdot(y_{1}\otimes y_{2})=x_{1}\cdot_{1} y_{1}\otimes x_{2}\cdot_{2}y_{2},
\end{equation}
\begin{equation}\label{eq:thm:tensor algebra2}
(x_{1}\otimes x_{2})\circ(y_{1}\otimes y_{2})=x_{1}\circ_{1}y_{1}\otimes x_{2}\cdot_{2}y_{2}+x_{1}\cdot_{1}y_{1}\otimes x_{2}\circ_{2}y_{2},
\end{equation}
for all $x_{1},y_{1}\in A_{1}, x_{2},y_{2}\in A_{2}$. Then
$(A_{1}\otimes A_{2},\cdot, \circ)$ is an anti-pre-Lie Poisson
algebra.
\end{thm}

\begin{proof}
It is known that $(A_{1}\otimes A_{2},\cdot)$ is a commutative
associative algebra. Let $x_{1},y_{1},z_{1}\in A_{1}$, and
$x_{2},y_{2},z_{2}\in A_{2}$. 
Then we have
{\small\begin{eqnarray*}
&&(x_{1}\otimes x_{2})\circ((y_{1}\otimes y_{2})\circ(z_{1}\otimes z_{2}))-(y_{1}\otimes y_{2})\circ((x_{1}\otimes x_{2})\circ(z_{1}\otimes z_{2}))-[y_{1}\otimes y_{2}, x_{1}\otimes x_{2}]\circ(z_{1}\otimes z_{2})\\
&&=x_{1}\circ_{1}(y_{1}\circ_{1} z_{1})\otimes x_{2}\cdot_{2} y_{2}\cdot_{2}z_{2}+x_{1}\cdot_{1}(y_{1}\circ_{1} z_{1})\otimes x_{2}\circ_{2}(y_{2}\cdot_{2}z_{2})+x_{1}\circ_{1}(y_{1}\cdot_{1}z_{1})\otimes x_{2}\cdot_{2}(y_{2}\circ_{2} z_{2})\\
&&\hspace{0.4cm}+x_{1}\cdot_{1}y_{1}\cdot_{1}z_{1}\otimes x_{2}\circ_{2}(y_{2}\circ_{2}z_{2})-
y_{1}\circ_{1}(x_{1}\circ_{1} z_{1})\otimes y_{2}\cdot_{2} x_{2}\cdot_{2}z_{2}-y_{1}\cdot_{1}(x_{1}\circ_{1} z_{1})\otimes y_{2}\circ_{2}(x_{2}\cdot_{2}z_{2})\\
&&\hspace{0.4cm}-y_{1}\circ_{1}(x_{1}\cdot_{1}z_{1})\otimes y_{2}\cdot_{2}(x_{2}\circ_{2} z_{2})-y_{1}\cdot_{1}x_{1}\cdot_{1}z_{1}\otimes y_{2}\circ_{2}(x_{2}\circ_{2}z_{2})
-[y_{1},x_{1}]_{1}\circ_{1}z_{1}\otimes x_{2}\cdot_{2} y_{2}\cdot_{2} z_{2}\\
&&\hspace{0.4cm}-[y_{1},x_{1}]_{1}\cdot_{1}z_{1}\otimes (x_{2}\cdot_{2} y_{2})\circ_{2} z_{2}-(x_{1}\cdot_{1} y_{1})\circ_{1} z_{1}\otimes [y_{2},x_{2}]_{2}\cdot_{2}z_{2}-x_{1}\cdot_{1} y_{1}\cdot_{1} z_{1}\otimes [y_{2},x_{2}]_{2}\circ_{2}z_{2}\\
&&=x_{1}\cdot_{1}(y_{1}\circ_{1} z_{1})\otimes x_{2}\circ_{2}(y_{2}\cdot_{2}z_{2})+x_{1}\circ_{1}(y_{1}\cdot_{1}z_{1})\otimes x_{2}\cdot_{2}(y_{2}\circ_{2} z_{2})-y_{1}\cdot_{1}(x_{1}\circ_{1} z_{1})\otimes y_{2}\circ_{2}(x_{2}\cdot_{2}z_{2})\\
&&\hspace{0.4cm}-y_{1}\circ_{1}(x_{1}\cdot_{1}z_{1})\otimes y_{2}\cdot_{2}(x_{2}\circ_{2} z_{2})+[x_{1},y_{1}]_{1}\cdot_{1}z_{1}\otimes(x_{2}\cdot_{2} y_{2})\circ_{2}z_{2}+(x_{1}\cdot_{1}y_{1})\circ_{1}z_{1}\otimes [x_{2},y_{2}]_{2}\cdot_{2} z_{2}.
\end{eqnarray*}}
By Eqs.~(\ref{eq:defi:anti-pre-Lie Poisson1}) and
(\ref{eq:defi:anti-pre-Lie Poisson2}), we have
{\small\begin{eqnarray*}
&&[x_{1}, y_{1}]_{1}\cdot_{1} z_{1}\otimes (x_{2}\cdot_{2} y_{2})\circ_{2} z_{2}=\dfrac{1}{2}y_{1}\cdot_{1}(x_{1}\circ_{1} z_{1})\otimes (x_{2}\cdot_{2} y_{2})\circ_{2} z_{2}-\dfrac{1}{2}x_{1}\cdot_{1}(y_{1}\circ_{1} z_{1})\otimes y_{2}\circ_{2}(x_{2}\cdot_{2} z_{2}),\\
&&\dfrac{1}{2}y_{1}\cdot_{1}(x_{1}\circ_{1} z_{1})\otimes (x_{2}\cdot_{2} y_{2})\circ_{2} z_{2}-y_{1}\cdot_{1}(x_{1}\circ_{1} z_{1})\otimes y_{2}\circ_{2}(x_{2}\cdot_{2}z_{2})=-\dfrac{1}{2}y_{1}\cdot_{1}(x_{1}\circ_{1} z_{1})\otimes x_{2}\cdot_{2} (y_{2}\circ_{2} z_{2}),\\
&&-\dfrac{1}{2}x_{1}\cdot_{1}(y_{1}\circ_{1} z_{1})\otimes (x_{2}\cdot_{2} y_{2})\circ_{2} z_{2}+x_{1}\cdot_{1}(y_{1}\circ_{1} z_{1})\otimes x_{2}\circ_{2}(y_{2}\cdot_{2}z_{2})=\dfrac{1}{2}x_{1}\cdot_{1}(y_{1}\circ_{1} z_{1})\otimes y_{2}\cdot_{2} (x_{2}\circ_{2} z_{2}).\\
\end{eqnarray*}}
Thus
{\small\begin{eqnarray*}
&&[x_{1}, y_{1}]_{1}\cdot_{1} z_{1}\otimes (x_{2}\cdot_{2} y_{2})\circ_{2} z_{2}-y_{1}\cdot_{1}(x_{1}\circ_{1} z_{1})\otimes y_{2}\circ_{2}(x_{2}\cdot_{2}z_{2})+x_{1}\cdot_{1}(y_{1}\circ_{1} z_{1})\otimes x_{2}\circ_{2}(y_{2}\cdot_{2}z_{2})\\
&&=\dfrac{1}{2}x_{1}\cdot_{1}(y_{1}\circ_{1} z_{1})\otimes y_{2}\cdot_{2} (x_{2}\circ_{2} z_{2})-\dfrac{1}{2}y_{1}\cdot_{1}(x_{1}\circ_{1} z_{1})\otimes x_{2}\cdot_{2} (y_{2}\circ_{2} z_{2}),
\end{eqnarray*}}
and similarly
{\small\begin{eqnarray*}
&&(x_{1}\cdot_{1} y_{1})\circ_{1} z_{1}\otimes [x_{2}, y_{2}]_{2}\cdot_{2} z_{2}-y_{1}\circ_{1}(x_{1}\cdot_{1} z_{1})\otimes y_{2}\cdot_{2}(x_{2}\circ_{2} z_{2})+x_{1}\circ_{1}(y_{1}\cdot_{1} z_{1})\otimes x_{2}\cdot_{2}(y_{2}\circ_{2} z_{2})\\
&&=\dfrac{1}{2}y_{1}\cdot_{1}(x_{1}\circ_{1} z_{1})\otimes x_{2}\cdot_{2} (y_{2}\circ_{2} z_{2})-\dfrac{1}{2}x_{1}\cdot_{1}(y_{1}\circ_{1} z_{1})\otimes y_{2}\cdot_{2} (x_{2}\circ_{2} z_{2}).
\end{eqnarray*}}
Hence Eq.~(\ref{eq:defi:anti-pre-Lie algebras1}) holds on
$A_{1}\otimes A_{2}$. By a similar proof,
Eqs.~(\ref{eq:defi:anti-pre-Lie algebras2}),
(\ref{eq:defi:anti-pre-Lie Poisson1}) and
(\ref{eq:defi:anti-pre-Lie Poisson2})  hold on $A_{1}\otimes
A_{2}$. Thus  $(A_{1}\otimes A_{2},\cdot, \circ)$ is an
anti-pre-Lie Poisson algebra.
\end{proof}

\begin{cor}
Let $(A_{1},\cdot_{1},\circ_{1})$ and
$(A_{2},\cdot_{2},\circ_{2})$ be two admissible Novikov-Poisson
algebras. Define two bilinear operations $\cdot$ and $\circ$ on $A_{1}\otimes
A_{2}$ respectively by Eqs.~(\ref{eq:thm:tensor algebra1})
and (\ref{eq:thm:tensor algebra2}). Then $(A_{1}\otimes
A_{2},\cdot, \circ)$ is an admissible Novikov-Poisson algebra.
\end{cor}

\begin{proof} One can prove that
$(A_{1}\otimes A_{2},\circ)$ is an admissible Novikov algebra
directly or as follows. Let $(A_{1},\cdot_{1},\star_{1})$ and
$(A_{2},\cdot_{2},\star_{2})$ be two Novikov-Poisson algebras such
that
 $(A_{1},\circ_{1})$ and $(A_{2},\circ_{2})$ are  the $2$-algebras of
 $(A_{1},\star_{1})$ and $(A_{2},\star_{2})$ respectively. By \cite{Xu1996}, there is a Novikov-Poisson algebra $(A_{1}\otimes A_{2},\cdot,\star)$ in which
$(A_1\otimes A_2, \cdot)$ is
given by Eq.~(\ref{eq:thm:tensor algebra1}) and $(A_1\otimes A_2,\star)$ is given by
\begin{equation}\label{eq:Novikov Poisson algebra}
(x_{1}\otimes x_{2})\star(y_{1}\otimes y_{2})=x_{1}\star_{1}y_{1}\otimes x_{2}\cdot_{2}y_{2}+x_{1}\cdot_{1} y_{1}\otimes x_{2}\star_{2}y_{2}, \forall x_{1},y_{1}\in A_{1}, x_{2},y_{2}\in A_{2}.
\end{equation}
Let $(A_{1}\otimes A_{2}, \circ)$ be the $2$-algebra of
$(A_{1}\otimes A_{2}, \star)$. Then by Proposition \ref{pro:from
admissible Novikov Poisson to Novikov Poisson},  $(A_{1}\otimes
A_{2},\cdot, \circ)$ is an admissible Novikov-Poisson algebra in
which $(A_1\otimes A_2,\cdot)$ is given by Eq.~(\ref{eq:thm:tensor
algebra1}) and $(A_1\otimes A_2,\circ)$ is given by
\begin{eqnarray*}
&&(x_{1}\otimes x_{2})\circ (y_{1}\otimes y_{2})=(x_{1}\otimes x_{2})\star(y_{1}\otimes y_{2})+2(y_{1}\otimes y_{2})\star(x_{1}\otimes x_{2})\\
&&=x_{1}\star_{1}y_{1}\otimes x_{2}\cdot_{2}y_{2}+x_{1}\cdot_{1}y_{1}\otimes x_{2}\star_{2}y_{2}+2y_{1}\star_{1}x_{1}\otimes x_{2}\cdot_{2}y_{2}+2x_{1}\cdot_{1}y_{1}\otimes y_{2}\star_{2}x_{2}\\
&&=x_{1}\circ_{1}y_{1}\otimes x_{2}\cdot_{2}y_{2}+x_{1}\cdot_{1}y_{1}\otimes x_{2}\circ_{2}y_{2},
\end{eqnarray*}
for all $x_1, y_1\in A_1$ and $x_2,y_2\in A_2$, which is exactly Eq.~(\ref{eq:thm:tensor algebra2}).
\end{proof}

\bigskip

 \noindent
 {\bf Acknowledgements.}  This work is supported by
NSFC (11931009), the Fundamental Research Funds for the Central
Universities and Nankai Zhide Foundation.


\begin{thebibliography}{99}








\bibitem{AM}A. Agore and G. Militaru, Jacobi and Poisson
algebras, {\it J. Noncomm. Geom.}  {9} (2015) 1295-1342.


\bibitem{Bai2004} C. Bai, Left-symmetric algebras from linear functions,
\textit{J. Algebra} 281 (2004) 651-665.

\bibitem{Bai2006} C. Bai, A further study on non-abelian phase spaces: Left-symmetric algebraic approach and related geometry, \textit{Rev. Math. Phys.} 18 (2006) 545-564.



\bibitem{Bai2007} C. Bai, A unified algebraic approach to the classical Yang-Baxter equation, \textit{ J. Phys. A: Math. Theor}. 40 (2007)







\bibitem{Bai2020} C. Bai, R. Bai, L. Guo and Y. Wu, Transposed Poisson algebras, Novikov-Poisson algebras, and 3-Lie algebras,
arXiv: 2005.01110

\bibitem{BLP} C. Bai, H. Li and Y. Pei,  $\phi_\epsilon$-coordinated modules for vertex algebras, \textit{J. Algebra} 426 (2015) 211-242.


\bibitem{Bai2001} C. Bai and D. Meng, The classification of Novikov algebras in low dimensions,
\textit{J. Phys. A: Math. Gen.} 34 (2001) 1581-1594.




\bibitem{Bal} A.A. Balinskii and S.P. Novikov, Poisson brackets of hydrodynamic type, Frobenius algebras and Lie
algebras, \textit{Soviet Math. Dokl.} 32 (1985) 228-231.


\bibitem{Burde1998} D. Burde, Simple left-symmetric algebras with
solvable Lie algebra, {\it Manuscripta Math.} 95 (1998) 397-411.

\bibitem{Bur} D. Burde, Left-symmetric algebras, or pre-Lie algebras in geometry and physics,  \textit{Cent. Eur. J. Math}. 4 (2006) 323-357.


\bibitem{CK} N. Cantarini and V. Kac, Classification of linearly compact simple Jordan and generalized Poisson superalgebras, \textit{ J. Algebra} { 313} (2007) 100-124.



\bibitem{Chu} B.Y. Chu, {Symplectic} homogeneous spaces, \textit{Trans. Amer. Math.Soc.} 197 (1974) 145-159.



\bibitem{Dot} V. Dotsenko, Algebraic structures of $F$-manifolds via pre-Lie algebras.  \textit{Ann. Mat. Pura Appl.} 198 (2019) 517-527.


\bibitem{Dzh09} A. Dzhumadil'daev, Algebras with skew-symmetric identity of degree 3, \textit{J. Math. Sci.} 161 (2009)  11-30.

\bibitem{Dzh092} A. Dzhumadil'daev and  A. Bakirova, Simple two-sided anti-Lie-admissible algebras, \textit{J. Math. Sci.} 161 (2009)  31-36.


\bibitem{Dzh} A. Dzhumadil'daev and P. Zusmanovich, Commutative 2-cocycles on Lie algebras, \textit{J. Algebra} 324 (2010) 732-748.

\bibitem{Fil1989} V.T. Filippov, A class of simple nonassociative algebras,
\textit{Mat. Zametki}  45 (1989) 101-105.

\bibitem{Fil} V.T. Filippov, Lie algebras satisfying identities of degree 5,
\textit{Algebra and Logic} 34 (1996) 379-394.

\bibitem{Fro} G. Frobenius, Theorie der hyperkomplexen Gr\"{o}{\ss}en,\textit{ Sitzber. K\"{o}niglich Preuss. Akad.
Wiss. Berlin} 1903 (1903) 504-537; Theorie der hyperkomplexen Gr\"{o}{\ss}en. II; \textit{ibid}
1903 (1903) 634-645.



\bibitem{Ger} M. Gerstenhaber,
The cohomology structure of an associative ring, {\it Ann. Math.} 78
(1963) 267-288.

\bibitem{Gel} I.M. Gel'fand and I.Ya. Dorfman, Hamiltonian operators and algebraic structures related to them, \textit{Funct. Anal. Appl.} 13 (1979) 248-262.




\bibitem{GK} L. Guo and W. Keigher, On differential Rota-Baxter algebras, {\it J. Pure Appl. Algebra} 212 (2008) 522-540.


\bibitem{Ko}
J.-L. Koszul, Domaines born\'es homog\`enes et orbites de groupes de
transformations affines, {\it Bull. Soc. Math. France} 89 (1961) 515-533.

\bibitem{Ku1} B.A. Kupershmidt, Non-abelian phase
spaces, {\it J. Phys. A: Math. Gen.} 27 (1994) 2801-2810.


\bibitem{Kup} B.A. Kupershmidt, What a classical $r-$matrix really is, \textit{J. Nonlinear Math. Phys.} 6 (1999) 448-488.




\bibitem{LLB} Y. Lin, X. Liu and C. Bai, Differential antisymmetric infinitesimal
bialgebras, coherent derivations and Poisson bialgebras, arXiv: 2207.00390.


\bibitem{Oku} S. Okubo and N. Kamiya,  Jordan-Lie superalgebra and Jordan-Lie triple system, \textit{J. Algebra} 198 (1997) 388-411.

\bibitem{MZ} C. Martnez, E. Zelmanov, Brackets, superalgebras and
spectral gap, \textit{S\~ao Paulo J. Math. Sci.} 13 (2019)
112-132.

\bibitem{Me}
A. Medina, Flat left-invariant connections adapted to the
automorphism structure of a Lie group, {\it J. Diff. Geom.} 16
(1981) 445-474.

\bibitem{SXZ} Y. Su, X. Xu and H. Zhang, Derivation-simple algebras and structures of Lie algebras of Witt type, \textit{J. Algebra} 233 (2000) 642-662.

\bibitem{SS} S.I. Svinolupov and V.V. Sokolov, Vector-matrix
generalizations of classical integrable equations, {\it Theoret.
and Math. Phys.} 100 (1994) 959-962.


\bibitem{Xu1996} X. Xu, On simple Novikov algebras and their irreducible modules, \textit{J. Algebra} 185 (1996) 905-934.

\bibitem{Xu1997} X. Xu, Novikov-Poisson algebras, \textit{J. Algebra} 190 (1997) 253-279.

\bibitem{Xu} X. Xu, New generalized simple Lie algebras of Cartan type over a field with characteristic zero, \textit{J. Algebra} 224 (2000) 23-58.

\bibitem{V} E.B. Vinberg, Convex homogeneous
cones, {\it Transl. of Moscow Math. Soc.} 12 (1963) 340-403.


\bibitem{Zu} P. Zusmanovich, The second homology group of current Lie algebras, {\textit{Ast\'{e}risque}} 226 (1994) 435-452.

\end{thebibliography}
\end{document}